\documentclass[a4paper,reqno]{amsart}
\usepackage{geometry}
\title[Computations of Spin-Sp(4), Spin-SU(8), and Spin-Spin(16) bordism groups]{Computations of Spin-Sp(4), Spin-SU(8), and Spin-Spin(16) bordism groups in dimensions up to 7}
\author{Naoki Kuroda}
\makeatletter
\@namedef{subjclassname@2020}{\textup{2020} Mathematics Subject Classification}
\makeatother
\address{GRADUATE SCHOOL OF MATHEMATICAL SCIENCES, THE
UNIVERSITY OF TOKYO, 3-8-1 KOMABA, MEGURO-KU, TOKYO, 153-8914,
JAPAN}
\email{kuronao0402@g.ecc.u-tokyo.ac.jp}
\usepackage{amsmath}
\usepackage{amssymb}
\usepackage{amsfonts}
\usepackage{mathrsfs}
\usepackage{amsthm}
\usepackage{bm}
\usepackage[dvipsnames]{xcolor}
\usepackage[pdftex]{graphicx}
\usepackage{wrapfig}
\usepackage[numbers]{natbib}
\usepackage{color}
\usepackage[labelsep=colon]{caption}
\usepackage[labelformat=empty,subrefformat=parens]{subcaption}
\usepackage{multicol}
\usepackage[all]{xy}
\usepackage{tikz-cd}
\usepackage{booktabs}
\usepackage[pdftex]{hyperref}
\usepackage{cleveref}
\usepackage{pgfplots}
\pgfplotsset{compat=1.15}
\usepackage{tikz}
\usepackage{adamsmacros}
\usepackage{subcaption}
\usetikzlibrary{arrows.meta, decorations.pathmorphing, positioning}
\usepackage{tikz-cd}

\newtheorem{thm}{Theorem}[section]
\newtheorem{prop}[thm]{Proposition}
\newtheorem{lemma}[thm]{Lemma}

\newtheorem*{mthm*}{Main Theorem}
\newtheorem*{thm*}{Theorem}
\newtheorem*{prop*}{Proposition}
\theoremstyle{definition}

\newtheorem{remark}[thm]{Remark}

\newcommand{\Z}{{\mathbb Z}}

\newcommand{\Spin}{\mathrm{Spin}}

\newcommand{\pt}{\mathrm{pt}}

\newcommand{\CP}{\mathbb{CP}}
\newcommand{\HP}{\mathbb{HP}}

\begin{document}
\begin{abstract}
We investigate the structure of Spin-$G$ bordism groups, focusing on the interplay between Spin and additional twisting symmetries such as \( Sp(4) \), \( SU(8) \) and $Spin(16)$. Using techniques from spectral sequences, obstruction theory, and cohomology operations, we compute explicit generators for the Spin-$G$ bordism groups in dimensions up to 7.
\end{abstract}
\maketitle
\section{Introduction}

\subsection{\texorpdfstring{The relation between Spin-$G$ bordism and quantum theory}{The relation between Spin-G bordism and quantum theory}}

Bordism theories provide a foundational framework in algebraic topology, linking geometric, algebraic, and homotopical structures. Among these, spin bordism plays a crucial role in both pure mathematics and theoretical physics. In string theory, spin structures are essential for defining consistent backgrounds for fermionic fields \cite{Witten1985}. Twisted spin bordism groups further refine this framework via added additional symmetry constraints that naturally emerge in anomaly cancellation mechanisms \cite{FreedHopkins2016}.

The twisted spin bordism groups \( \Omega_*^{\Spin\text{-}G} \), where \( G \) is a Lie group (possibly discrete) and we follow the notation of Debray-Yu \cite{DY24}, extend the standard spin bordism groups by encoding supplementary symmetries that frequently arise in the field of string theory. Twisted spin bordism thus provides a powerful framework for analyzing the existence of anomalies in string theory, as discussed in \cite{DDHM23}.

\subsection{\texorpdfstring{Overview of Spin-$G$ bordism groups}{Overview of Spin-G bordism groups}}\label{subsec:Spin-G}

The notion of twisted spin bordism groups or Spin-\( G \) bordism groups \( \Omega_{\ast}^{\Spin\text{-}G} \) generalizes spin bordism by incorporating additional symmetry constraints. Let \( G \) and \( H \) be Lie groups, and let there be a double cover \( \{\pm 1\} \to G \to H \). The tangential structure is induced by the following sequence:
\begin{equation}\label{eq:Spin-G}
    \Spin\text{-}G := \Spin \times_{\{\pm 1\}} G \to SO \times H \to SO \to O,
\end{equation}
is referred to as a \( \Spin\text{-}G \) structure. 

In this paper, we consider the cases \[ (G, H) = (Sp(4), Sp(4)/\mathbb{Z}_2), 
 (SU(8), SU(8)/\mathbb{Z}_2) , 
 (Spin(16), Ss(16)) .
\]

In theoretical physics, Spin-\( G \) bordism groups are particularly relevant in string theory, where they appear in the classification of consistent backgrounds and in the study of anomalies. For instance, the cancellation of global anomalies in certain string theories requires an understanding of the bordism groups associated with specific symmetry constraints. 
In recent years, a program to compute the Spin-$G$ bordism group for Lie groups $G$ that appear as symmetries in string theory and supergravity theory has been progressing. In \cite{DDHM23}, the Spin-$G$ bordism group is computed for $(G, H) = ( Mp(2, \mathbb{Z}), SL(2, \Z))$ and $(GL^{+}(2, \mathbb{Z}), GL(2, \Z))$.

The computation of Spin-\( G \) bordism groups often involves advanced tools such as spectral sequences, obstruction theory, and cohomology operations. These techniques enable explicit calculations of low-dimensional groups and the identification of their generators. In this paper, we determine the structure of \( \Omega_{\ast \leq 7}^{\Spin\text{-}Sp(4)} \), \( \Omega_{\ast \leq 7}^{\Spin\text{-}SU(8)} \), and \( \Omega_{\ast \leq 7}^{\Spin\text{-}Spin(16)} \) as \( \mathbb{Z} \)-modules and explicitly describe the Spin-$G$ manifolds generating each module.

Spin-\( G \) bordism theory thus serves as a bridge between topology and physical applications, offering profound insights into the interplay between geometry, symmetry, and physics.

\subsection{\texorpdfstring{Introducing three Spin-$G$ bordism groups}{Introducing three Spin-G bordism groups}}

In the following, \( \mathbb{Z}_2 \) denotes the cyclic group of order \( 2 \). In \cite{DY24}, the computation of the \( \Spin\text{-}SU(8) \) bordism group was carried out based on the observation that the symmetry of 4d \( \mathcal{N} = 8 \) supergravity is given by \( E_{7(7)} \), which denotes the split real form of the exceptional Lie group, and that the maximal compact subgroup of \( E_{7(7)} \) is \( SU(8)/\mathbb{Z}_2 \). 
In Appendix B of Polchinski \cite{Po98}, a table summarizing the symmetries of low-energy supergravity theories states that the symmetries of the supergravity theories of 5d \( \mathcal{N} = 7 \), 4d \( \mathcal{N} = 8 \), and 3d \( \mathcal{N} = 9 \) are \( E_{6(6)} \), \( E_{7(7)} \), and \( E_{8(8)} \), respectively. Furthermore, the maximal compact subgroups of \( E_{6(6)} \), \( E_{7(7)} \), and \( E_{8(8)} \) are \( Sp(4)/\mathbb{Z}_2 \), \( SU(8)/\mathbb{Z}_2 \), and \( SemiSpin(16) \), respectively \cite{MCS}. 
Motivated by these observations and inspired by Debray-Yu \cite{DY24}, we aim to compute the \( \Spin\text{-}Sp(4) \), \( \Spin\text{-}SU(8) \), and \( \Spin\text{-}Spin(16) \) bordism groups.

Our main results are summarized in the following theorem.

\begin{thm}\label{thm:generator}
The three Spin-$G$ bordism groups and their respective generators in dimensions up to \( 7 \) are as follows:

\begin{table}[htbp]
\centering
\begin{tabular}{c c c}
\toprule
\toprule
$k$ & $\Omega_k^{\Spin\text{-}Sp(4)}$ & Generators \\
\midrule
$0$ & $\Z$ & $\pt$  \\
$1$ & $0$ &  \\
$2$ & $0$ &  \\
$3$ & $0$ &  \\
$4$ & $\Z\oplus \Z$ & $(\HP^1, \CP^2)$ \\
$5$ & $\Z_2\oplus \Z_2$ & $(\HP^1\times S^1, SU(3)/SO(3))$ \\
$6$ & $\Z_2\oplus \Z_2$ & $(\HP^1\times S^1\times S^1, \CP^2\times \CP^1)$ \\
$7$ & $0$ &  \\
\bottomrule
\end{tabular}
\label{tab:twistedSp}
\end{table}

\begin{table}[htbp]
\centering
\begin{tabular}{c c c}
\toprule
\toprule
$k$ & $\Omega_k^{\Spin\text{-}SU(8)}$ & Generators \\
\midrule
$0$ & $\Z$ & $\pt$  \\
$1$ & $0$ &  \\
$2$ & $0$ &  \\
$3$ & $0$ &  \\
$4$ & $\Z\oplus \Z$ & $(\HP^1, \CP^2)$ \\
$5$ & $\Z_2$ & $SU(3)/SO(3)$ \\
$6$ & $\Z\oplus \Z_2$ & $(\CP^1\times \CP^1\times \CP^1, \CP^2\times \CP^1)$ \\
$7$ & $0$ &  \\
\bottomrule
\end{tabular}
\label{tab:twistedSU}
\end{table}

\begin{table}[htbp]
\centering
\begin{tabular}{c c c}
\toprule
\toprule
$k$ & $\Omega_k^{\Spin\text{-}Spin(16)}$ & Generators \\
\midrule
$0$ & $\Z$ & $\pt$  \\
$1$ & $0$ &  \\
$2$ & $0$ &  \\
$3$ & $0$ &  \\
$4$ & $\Z\oplus \Z$ & $(\HP^1, \CP^2)$ \\
$5$ & $\Z_2$ & $SU(3)/SO(3)$ \\
$6$ & $\Z_2$ & $\CP^2\times \CP^1$ \\
$7$ & $0$ &  \\
\bottomrule
\end{tabular}
\label{tab:twistedSpin}
\caption{The specific Spin-$G$ structures of the manifolds mentioned here will be discussed in Section 4.}
\end{table}
\end{thm}

\section*{Acknowledgements}
The author thanks Takuya Sakasai for his support and helpful comments. This research was supported by FMSP, WINGS Program, the University of Tokyo.

\section{Mathematical Background}

\subsection{The Leray-Serre spectral sequence}

The Leray-Serre spectral sequence is a powerful computational tool in algebraic topology, particularly in the study of fibrations. Given a ring $R$ and a fibration
\begin{equation*}
    F \to E \to B,
\end{equation*}
where \( F \) is the fiber, \( E \) is the total space, and \( B \) is the base space, the spectral sequence provides a bridge between the cohomology of \( F \), \( B \), and \( E \).

The \( E_2 \)-page of the spectral sequence is given by
\begin{equation}
    E_2^{p,q} = H^p(B; H^q(F; R)),
\end{equation}
where \( H^q(F; R) \) denotes the cohomology of the fiber \( F \) with coefficients in $R$, viewed as a local coefficient system over \( B \). The spectral sequence converges to the cohomology of the total space \( E \), that is,
\begin{equation}
    E_\infty^{p,q} \implies H^{p+q}(E; R).
\end{equation}

This spectral sequence is particularly useful for computing cohomology when the fibration has a simple or well-understood fiber \( F \) and base \( B \). For example, it has been extensively applied to the study of homogeneous spaces, classifying spaces, and loop spaces (see, e.g., \cite{Mc01, Ha02}).

The differentials \( d_r: E_r^{p,q} \to E_r^{p+r,q-r+1} \) encode higher-order interactions between the cohomology of the fiber and the base. These differentials are often determined by additional geometric or algebraic structures associated with the fibration.

For a detailed exposition of the Leray-Serre spectral sequence, we refer the reader to McCleary's book on spectral sequences \cite{Mc01} and Hatcher's foundational text on algebraic topology \cite{Ha02}.

\subsection{The Adams spectral sequence}\label{sec:ASS}

The Adams spectral sequence (ASS) is a fundamental computational tool in stable homotopy theory. It provides a systematic approach to understanding stable homotopy groups of spheres and other spectra by connecting algebraic and topological invariants.

The classical Adams spectral sequence arises from a filtration of spectra based on their connectivity. For a spectrum \( X \), the spectral sequence is given by
\begin{equation}\label{eq:ASS}
    E_2^{s,t} = \mathrm{Ext}_{\mathcal{A}}^{s,t}(H^\ast(X; \mathbb{Z}_p), \mathbb{Z}_p) \implies \pi_{t-s}(X)_{\widehat{p}},
\end{equation}
where \( \mathcal{A} \) is the mod \( p \) Steenrod algebra, \( H^\ast(X; \mathbb{Z}_p) \) denotes the mod \( p \) cohomology of \( X \), and \( \pi_\ast(X)_{\widehat{p}} \) represents the \( p \)-local stable homotopy groups of \( X \). The \( E_2 \)-page is computed using the cohomology of the Steenrod algebra, and the spectral sequence converges under suitable conditions. For a foundational introduction, we refer to Adams' seminal text \cite{Adams1974} and Ravenel's comprehensive treatment \cite{Ravenel1986}.

In the following, we describe a method to compute the 2-torsion part of Spin bordism groups using the Adams spectral sequence. By employing the Anderson-Brown-Peterson decomposition \cite{ABP67}, the Spin bordism group is expressed in terms of connective \( ko \)-homology as follows:
\begin{equation}
\label{ABP}
    \Omega_n^{\Spin}(X)_{\widehat{2}} = ko_n(X)_{\widehat{2}} \, \oplus \,  ko_{n-8}(X)_{\widehat{2}} \, \oplus \,  ko_{n-10}\langle 2 \rangle(X)_{\widehat{2}} \oplus \,\dots\,.
\end{equation}

For connective \( ko \)-homology, the following relation holds:
\begin{equation}
    H^{*}(ko ; \mathbb{Z}_2) \cong \mathcal{A} \otimes_{\mathcal{A}_1} \mathbb{Z}_2\,.
\end{equation}
Using this result, by performing a change of coefficient rings in the Adams spectral sequence \eqref{eq:ASS}, we obtain:
\begin{equation}
    E_{2}^{s,t} = \text{Ext}^{s,t}_{\mathcal{A}_1}(H^*(X; \mathbb{Z}_2), \mathbb{Z}_2) \Rightarrow ko_{t-s}(X)_{\widehat{2}}\,.
\end{equation}
This indicates that studying the structure of the \( \mathcal{A}_1 \)-action on \( H^*(X; \mathbb{Z}_2) \) provides a foothold for computing \( \Omega_n^{\Spin}(X)_{\widehat{2}} \).

For a more detailed discussion of the theoretical background and specific examples, we refer the reader to Beaudry-Campbell \cite[Section 4]{BC18}, Debray et al. \cite[Sections 10 and 11]{DDHM23}, and Kneißl \cite[Section 2.1]{Kn24}.

\subsection{\texorpdfstring{The Spin-$G$ structure and methods for computing Spin-$G$ bordism groups}{The Spin-G structure and methods for computing Spin-G bordism groups}}

In Section \ref{subsec:Spin-G}, we introduced the concept of Spin-$G$ structure. Here, we provide a more detailed explanation.

Let \( G \) and \( H \) be Lie groups, and let there be a double cover \( \{\pm 1\} \to G \to H \). The tangential structure induced by \eqref{eq:Spin-G} is called a \(\Spin\text{-}G\) structure.
For example, when $G = H = S^1$, this corresponds to the \(\Spin^c\) structure.

In this paper, we provide a detailed explanation of the three Spin-$G$ bordism groups under consideration. First, the symplectic group \( Sp(n) \) is defined as 
\begin{equation}
    Sp(n) = U(n, \mathbb{H}) = \{A \in M(n, \mathbb{H}) \mid AA^{\ast} = I\}.
\end{equation}
The center of this Lie group, \( Z(Sp(n)) \), is \( \{\pm I\} \), and the group obtained by quotienting out this center is denoted in this paper by \( Sp(n)/\Z_2 \).

Next, \( SU(2n) \) contains \( -I \), and the group obtained by quotienting \( SU(2n) \) by \( \{\pm I\} \) is denoted by \( SU(2n)/\Z_2 \) in this paper.

Finally, it is known that the universal cover \( Spin(4n) \) of \( SO(4n) \) has a center \( Z(Spin(4n)) \) isomorphic to \( \Z_2 \oplus \Z_2 \). Among the three subgroups of \( Z(Spin(4n)) \) isomorphic to \( \Z_2 \), we consider the two that are not the kernel of the covering map \( Spin(4n) \to SO(4n) \). The Lie groups obtained by quotienting \( Spin(4n) \) by these subgroups are isomorphic to each other. They are called the Semi-Spin group and we denote them by \( Ss(4n)=SemiSpin(4n) \) \cite[Theorem 4.4 and Theorem 4.5]{MT91}

From these considerations, we obtain the following three double covers of Lie groups:
\begin{equation}\label{Sp(4)fibration}
    \{\pm 1\} \to Sp(4) \to Sp(4)/\Z_2,
\end{equation}
\begin{equation}\label{SU(8)fibration}
    \{\pm 1\} \to SU(8) \to SU(8)/\Z_2,
\end{equation}
\begin{equation}\label{Spin(16)fibration}
    \{\pm 1\} \to Spin(16) \to Ss(16).
\end{equation}

The corresponding \( \Spin\text{-}G \) structures for these double covers are denoted in this paper by \( \Spin\text{-}Sp(4) \) structure, \( \Spin\text{-}SU(8) \) structure, and \( \Spin\text{-}Spin(16) \) structure, respectively.

When a manifold \( M \) is equipped with a \(\Spin\text{-}G\) structure, the above map provides the orientation structure \( M \to BSO \) and a map \( f_M: M \to BH \) via the composition \( M \to B(\Spin\text{-}G) \to BSO \times BH \). Conversely, given an oriented manifold \( M \) and a map $ f_M: M \to BH $, the obstruction to inducing a \(\Spin\text{-}G\) structure is given by \( w_2(M) + f_M^{\ast} \zeta \). Here, \( \zeta \) is the element of \( H^2(BH; \mathbb{Z}_2) \) that corresponds to the central extension \( \{\pm 1\} \to G \to H \), as described in \cite[Lemma 3.9]{DY24}.
Moreover, if \( M \) is simply connected, the lift is uniquely determined, so there is a one-to-one correspondence between the homotopy class of \( f_M: M \to BH \) satisfying \( w_2(M) = f_M^{\ast} \zeta \) and the \(\Spin\text{-}G\) structure on \( M \).

Using the Pontrjagin-Thom construction, we obtain 
\begin{equation}
    \Omega_k^{\Spin\text{-}G} \cong \pi_k (MT(\Spin\text{-}G)).
\end{equation}
Here, if a vector bundle $V \to BH$ of rank $r_V$ satisfies $w_2(V) = \zeta$, the Madsen-Tillmann spectrum $MT(\Spin\text{-}G)$ is decomposed using the Thom space of the virtual vector bundle $V - r_V$ as 
\begin{equation}
    MT(\Spin\text{-}G) \sim MTSpin \wedge BH^{V - r_V}
\end{equation}
(see \cite[Corollary 3.4]{DDHM23}). From this, we rewrite 
\begin{equation}
    \Omega_k^{\Spin\text{-}G} \cong \pi_k(MT(\Spin\text{-}G)) \cong \pi_k(MTSpin \wedge BH^{V - r_V}) \cong \Omega_k^{\Spin}(BH^{V - r_V}).
\end{equation}

Thus, by using the previously described method for computing $\Omega_k^{\Spin}(X)$, we can compute the $\Spin\text{-}G$ bordism group.

In this calculation, we assume the existence of a vector bundle $V$ over $BH$ satisfying $w_2(V) = \zeta$. However, there are cases where such a $V$ does not exist. Examples of these are discussed in \cite{DY23}, and \( (G, H) = (Sp(4), Sp(4)/\mathbb{Z}_2) \), 
\( (SU(8), SU(8)/\mathbb{Z}_2) \), and 
\( (Spin(16), Ss(16)) \) are such cases. For these examples, it was previously thought that the above method could not be applied. However, Debray and Yu \cite[Theorem (2)]{DY23} show that the computation remains valid and can be carried out in the same way, even in the absence of this condition. 

In this paper, we adopt the computation method based on that result.

\section{The Calculation of the three bordism groups in dimensions up to 7}
\label{sec:calculation_bordism}

\subsection{The relations among the three bordism groups}

We begin by establishing the connections among the three $\Spin\text{-}G$ bordism groups that form the central theme of this paper. The symplectic group \( Sp(n) \) is naturally regarded as a subgroup of \( U(2n) = U(2n, \mathbb{C}) \) in the following manner. 
By interpreting \( \mathbb{H}^n \) as \( \mathbb{C}^{2n} \), an isometry of \( \mathbb{H}^n \) is viewed as an isometry of \( \mathbb{C}^{2n} \). This provides a natural inclusion map \( Sp(n) \to U(2n) \), whose image lies in \( SU(2n) \subset U(2n) \) \cite{MT91}.

Similarly, there exists a natural inclusion \( U(2n) \to SO(4n) \). Since \( SU(2n) \) is simply connected, the composition \( SU(2n) \to U(2n) \to SO(4n) \) lifts to the universal cover \( Spin(4n) \) of \( SO(4n) \). This yields an inclusion map \( g: SU(2n) \to Spin(4n) \), as shown in the commutative diagram below:

\begin{equation}
    \begin{tikzcd}
 & SU(2n) \arrow[r, dashed, "g"] \arrow[d, hook, "s"] & Spin(4n) \arrow[d, two heads, "t"] \\
Sp(n) \arrow[ru, hook, dashed, "f"] \arrow[r, hook, "f'"] & U(2n) \arrow[r, hook, "g'"] & SO(4n)
\end{tikzcd}
\end{equation}

Next, we examine the image of \( -I \in Sp(n) \) under the maps \( f \) and \( g \). Since \( -I \) maps to \( -I \) through the sequence \( Sp(n) \to U(2n) \to SO(4n) \), we have \( f(-I) = -I \), and \( g(-I) \) corresponds to the lift of \( -I \). By quotienting out \( -I \) in each case, we obtain the following sequence of Lie group homomorphisms:

\begin{equation}
    \begin{tikzcd}
Sp(n)/\mathbb{Z}_2 \arrow[r, "\widetilde{f}"] & SU(2n)/\mathbb{Z}_2 \arrow[r, "\widetilde{g}"] & Ss(4n)
\end{tikzcd}
\end{equation}

These homomorphisms allow us to construct maps between the associated fibrations:

\begin{equation}
    \begin{tikzcd}
Sp(4) \arrow[r] \arrow[d, "f"'] & Sp(4)/\mathbb{Z}_2 \arrow[r] \arrow[d, "\widetilde{f}"'] & B\mathbb{Z}_2 \arrow[d, "id"] \\
SU(8) \arrow[r] \arrow[d, "g"'] & SU(8)/\mathbb{Z}_2 \arrow[r] \arrow[d, "\widetilde{g}"'] & B\mathbb{Z}_2 \arrow[d, "id"] \\
Spin(16) \arrow[r] & Ss(16) \arrow[r] & B\mathbb{Z}_2
\end{tikzcd}
\end{equation}

By the naturality of these maps, we obtain homomorphisms between the bordism groups:

\begin{equation}\label{maps between three bordism groups}
    \Omega_k^{\Spin\text{-}Sp(4)} \to \Omega_k^{\Spin\text{-}SU(8)} \to \Omega_k^{\Spin\text{-}Spin(16)}.
\end{equation}

\begin{prop}\label{prop:5-equivalence}
The maps \eqref{maps between three bordism groups} are isomorphisms for \( k \leq 4 \).
\end{prop}

\begin{proof}
First, we show that \( f \) and \( g \) are 4-equivalences. By standard homotopy group calculations, we know:

\[
\pi_0(Sp(4)) \cong \pi_0(SU(8)) \cong \pi_0(Spin(16)) \cong \mathbb{Z},
\]
\[
\pi_1(Sp(4)) \cong \pi_1(SU(8)) \cong \pi_1(Spin(16)) \cong 0,
\]
\[
\pi_2(Sp(4)) \cong \pi_2(SU(8)) \cong \pi_2(Spin(16)) \cong 0,
\]
\[
\pi_3(Sp(4)) \cong \pi_3(SU(8)) \cong \pi_3(Spin(16)) \cong \Z,
\]
\[
\pi_4(SU(8)) \cong \pi_4(Spin(16)) \cong 0.
\]
It suffices to show that the induced maps
\[
f_{\ast}: \pi_3(Sp(4)) \to \pi_3(SU(8)) \quad \text{and} \quad g_{\ast}: \pi_3(SU(8)) \to \pi_3(Spin(16))
\]
are isomorphisms. Using the properties of classifying spaces and the Hurewitz theorem, this reduces to showing that
\[
(Bf)^{\ast}: H^4(BSU(8); \mathbb{Z}) \to H^4(BSp(4); \mathbb{Z})
\]
and
\[
(Bg)^{\ast}: H^4(BSpin(16); \mathbb{Z}) \to H^4(BSU(8); \mathbb{Z})
\]
are isomorphisms. 

\begin{equation}
    \begin{tikzcd}
 & BSU(8) \arrow[r, " Bg"] \arrow[d, " Bs"] & BSpin(16) \arrow[d, "Bt"] \\
BSp(4) \arrow[ru, "Bf"] \arrow[r, "Bf'"] & BU(8) \arrow[r, "Bg'"] & BSO(16)
\end{tikzcd}
\end{equation}

In the commutative diagram above,  
\( Bt^{\ast}(p_1) = 2q_1 \) (where \( q_1 \) is the first spin class).  
Moreover, since  
\[
Bg'^{\ast}(1 - p_1 + p_2 - \ldots) = (1 + c_1 + c_2 + \ldots)(1 - c_1 + c_2 - \ldots),
\]
we have \( Bs^{\ast}(Bg'^{\ast}(p_1)) = Bs^{\ast}(c_1^2 - 2c_2 ) = -2c_2 \) and \( Bf'^{\ast}(Bg'^{\ast}(p_1)) = Bf'^{\ast}(c_1^2 - 2c_2) = -2p_1 \), as follows from the properties of characteristic classes.  

Thus, \( Bg^{\ast}(q_1) = -c_2 \) and \( Bf^{\ast}(c_2) = p_1 \), which shows that \( Bg^{\ast} \) and \( Bf^{\ast} \) map generators to generators.  
Therefore, \( f \) and \( g \) are 4-equivalence.

\begin{equation}
    \begin{tikzcd}
\Spin\times Sp(4) \arrow[r, ""] \arrow[d, "id\times f"'] & \Spin\text{-}Sp(4) \arrow[r, ""] \arrow[d, ""'] & B\Z_2 \arrow[d, "id"] \\
\Spin\times SU(8) \arrow[r, ""] \arrow[d, "id\times g"'] & \Spin\text{-}SU(8) \arrow[r, ""] \arrow[d, ""'] & B\Z_2 \arrow[d, "id"] \\
\Spin\times Spin(16) \arrow[r, ""] & \Spin\text{-}Spin(16) \arrow[r, ""] & B\Z_2
\end{tikzcd}
\end{equation}

By considering the Leray-Serre spectral sequence for the maps in the above fibration,  
\begin{equation}
    \Spin\text{-}Sp(4) \to \Spin\text{-}SU(8) \to \Spin\text{-}Spin(16)
\end{equation}
is 4-equivalence. Consequently, 
\[ 
MT(\Spin\text{-}Sp(4)) \to MT(\Spin\text{-}SU(8)) \to MT(\Spin\text{-}Spin(16)) 
\]
is (4+1) = 5-equivalence. This completes the proof.
\end{proof}

\subsection{No odd-prime torsion in the three bordism groups}

\begin{prop}
The bordism groups \( \Omega_{\ast}^{\Spin\text{-}Sp(4)} \), \( \Omega_{\ast}^{\Spin\text{-}SU(8)} \), and \( \Omega_{\ast}^{\Spin\text{-}Spin(16)} \) contain no \( p \)-torsion for odd primes \( p \).
\end{prop}

\begin{proof}
The result for \( \Omega_{\ast}^{\Spin\text{-}SU(8)} \) was proven in \cite{DY24}. The same argument applies to \( \Omega_{\ast}^{\Spin\text{-}Sp(4)} \). For \( \Omega_{\ast}^{\Spin\text{-}Spin(16)} \), the proof follows from the fact that \( H^{\ast}(BSpin(16); \mathbb{Z}) \) contains no \( p \)-torsion and has free components only in even dimensions \cite{Ko86}.
\end{proof}

\subsection{\texorpdfstring{Computation of cohomology rings with $\mathbb{Z}_2$ coefficients}{Computation of cohomology rings with Z2 coefficients}}

To perform the Adams computation, it is essential to first compute the cohomology rings and specify the Steenrod actions.

\begin{thm}\label{thm:cohomology ring}
The cohomology rings with $\mathbb{Z}_2$ coefficients for the spaces in question are given by:

\begin{align}
H^{\ast\leq 11}(B(Sp(4)/\mathbb{Z}_2); \mathbb{Z}_2) &= \mathbb{Z}_2[x_2, x_3, x_5, x_9, y_4]/((x_5 + x_2x_3)y_4), \\
H^{\ast\leq 11}(B(SU(8)/\mathbb{Z}_2); \mathbb{Z}_2) &= \mathbb{Z}_2[x_2, x_3, x_5, x_9, y_4, y_6, y_{10}]/(x_5y_4 + x_3y_6), \\
H^{\ast\leq 11}(BSs(16); \mathbb{Z}_2) = \mathbb{Z}_2&[x_2, x_3, x_5, x_9, y_4, y_6, y_7, y_{10}, y_{11}]/((x_5 + x_2x_3)y_4 + x_3y_6 + x_2y_7).
\end{align}
The Steenrod action on each element is described in the following tables.

\begin{table}[ht]
\centering
\begin{tabular}{|c|c|c|c|c|c|}
\hline
 & $Sq^1$ & $Sq^2$ & $Sq^3$ & $Sq^4$ & $Sq^5$ \\
\hline
$x_2$ & $x_3$ & ${x_2}^2$ & & & \\
\hline
$x_3$ & $0$ & $x_5$ & ${x_3}^2$ & & \\
\hline
$x_5$ & ${x_3}^2$ & $0$ & $0$ & $x_9$ & ${x_5}^2$ \\
\hline
$x_9$ & ${x_5}^2$ & $0$ & $0$ & & \\
\hline
$y_4$ & $0$ & $x_2y_4$ & $x_3y_4$ & ${y_4}^2$ & \\
\hline
\end{tabular}
\caption{Steenrod operations on elements of $H^{\ast\leq 11}(B(Sp(4)/\mathbb{Z}_2); \mathbb{Z}_2)$}
\end{table}

\begin{table}[ht]
\centering
\begin{tabular}{|c|c|c|c|c|c|c|}
\hline
 & $Sq^1$ & $Sq^2$ & $Sq^3$ & $Sq^4$ & $Sq^5$ & $Sq^6$ \\
\hline
$x_2$ & $x_3$ & ${x_2}^2$ & & & & \\
\hline
$x_3$ & $0$ & $x_5$ & ${x_3}^2$ & & & \\
\hline
$x_5$ & ${x_3}^2$ & $0$ & $0$ & $x_9$ & ${x_5}^2$ & \\
\hline
$x_9$ & ${x_5}^2$ & $0$ & $0$ & & & \\
\hline
$y_4$ & $0$ & $y_6$ & $x_3y_4$ & ${y_4}^2$ & & \\
\hline
$y_6$ & $x_3y_4$ & $0$ & $0$ & $y_{10}$ & $x_5y_6 + x_3{y_4}^2$ & ${y_6}^2$ \\
\hline
$y_{10}$ & $x_5y_6 + x_3{y_4}^2$ & ${y_6}^2$ & & & & \\
\hline
\end{tabular}
\caption{Steenrod operations on elements of $H^{\ast\leq 11}(B(SU(8)/\mathbb{Z}_2); \mathbb{Z}_2)$}
\end{table}

\begin{table}[ht]
\centering
\begin{tabular}{|c|c|c|c|c|c|c|}
\hline
 & $Sq^1$ & $Sq^2$ & $Sq^3$ & $Sq^4$ & $Sq^5$ & $Sq^6$ \\
\hline
$x_2$ & $x_3$ & ${x_2}^2$ & & & & \\
\hline
$x_3$ & $0$ & $x_5$ & ${x_3}^2$ & & & \\
\hline
$x_5$ & ${x_3}^2$ & $0$ & $0$ & $x_9$ & ${x_5}^2$ & \\
\hline
$x_9$ & ${x_5}^2$ & $0$ & $0$ & & & \\
\hline
$y_4$ & $0$ & $y_6$ & $y_7$ & ${y_4}^2$ & & \\
\hline
$y_6$ & $y_7$ & $0$ & $0$ & $y_{10}$ & $y_{11}$ & ${y_6}^2$ \\
\hline
$y_7$ & $0$ & $0$ & $0$ & $y_{11}$ & $0$ & \\
\hline
$y_{10}$ & $y_{11}$ & ${y_6}^2$ & & & & \\
\hline
$y_{11}$ & $0$ & & & & & \\
\hline
\end{tabular}
\caption{Steenrod operations on elements of $H^{\ast\leq 11}(BSs(16); \mathbb{Z}_2)$}
\end{table}
\end{thm}

\begin{proof}
The results for $BSs(16)$ are already known in \cite{Ta13}, \cite{Kn24}. For $B(Sp(4)/\mathbb{Z}_2)$ and $B(SU(8)/\mathbb{Z}_2)$, the computations will be carried out using known results and the following maps between fibrations.

\begin{equation}\label{three fibrations}
    \begin{tikzcd}
BSp(4) \arrow[r, ""] \arrow[d, "Bf"'] & B(Sp(4)/\mathbb{Z}_2) \arrow[r, ""] \arrow[d, "B\widetilde f"'] & B^2\mathbb{Z}_2 \arrow[d, "id"] \\
BSU(8) \arrow[r, ""] \arrow[d, "Bg"'] & B(SU(8)/\mathbb{Z}_2) \arrow[r, ""] \arrow[d, "B\widetilde g"'] & B^2\mathbb{Z}_2 \arrow[d, "id"] \\
BSpin(16) \arrow[r, ""] & BSs(16) \arrow[r, ""] & B^2\mathbb{Z}_2
\end{tikzcd}
\end{equation}

The cohomology rings of the spaces appearing in the left and right columns of \eqref{three fibrations}, with $\mathbb{Z}_2$ coefficients, are described in \cite{MT91} and are given as follows.

\begin{prop}
\begin{align*}
H^{\ast\leq 11}(BSp(4); \mathbb{Z}_2) &= \mathbb{Z}_2[p_1, p_2], \\
H^{\ast\leq 11}(BSU(8); \mathbb{Z}_2) &= \mathbb{Z}_2[c_2, c_3, c_4, c_5], \\
H^{\ast\leq 11}(BSpin(16); \mathbb{Z}_2) &= \mathbb{Z}_2[w_4, w_6, w_7, w_8, w_{10}, w_{11}], \\
H^{\ast\leq 11}(B^2\mathbb{Z}_2; \mathbb{Z}_2) &= \mathbb{Z}_2[\alpha, \beta, \gamma, \delta].
\end{align*}

Additionally, the following relations hold:
\[
Bf^{\ast}(c_{2i}) = p_i, \quad Bf^{\ast}(c_{2i+1}) = 0, \quad Bg^{\ast}(w_{2i}) = c_i, \quad Bg^{\ast}(w_{2i+1}) = 0.
\]

Here:
\begin{itemize}
    \item $p_i$ are the symplectic Pontryagin classes, reduced modulo $2$, and denoted by the same symbols.
    \item $c_i$ are the Chern classes, reduced modulo $2$, and denoted by the same symbols.
    \item $w_i$ are the Stiefel-Whitney classes, obtained by pulling back the map $BSpin(16) \to BSO(16)$, and denoted by the same symbols.
    \item $\alpha, \beta, \gamma, \delta$ are degree $2$, $3$, $5$, and $9$ elements, respectively.
\end{itemize}
\end{prop}

We consider the fibration 
\[
BSpin(16) \to BSs(16) \to B^2\mathbb{Z}_2.
\]
The cohomology ring \( H^{\ast\leq 11}(BSpin(16);\mathbb{Z}_2) \) is a polynomial ring generated by elements in degrees 4, 6, 7, 8, 10, and 11. Similarly, \( H^{\ast}(B^2\mathbb{Z}_2;\mathbb{Z}_2) \) is a polynomial ring generated by elements in degrees 2, 3, 5, and 9. 

From this, the following conclusions are drawn:

\begin{itemize}
    \item The elements \( \alpha \), \( \beta \), \( \gamma \), \( \delta \), \( w_4 \), \( w_6 \), \( w_7 \), \( w_{10} \), and \( w_{11} \) survive to the \( E_{\infty} \)-page of the Leray–Serre spectral sequence.
    \item The element \( w_8 \) is killed in the Leray–Serre spectral sequence by a differential map, pairing with a generator in degree 9.
\end{itemize}

We apply these results, in combination with the maps in the aforementioned fibration, to \( H^{\ast}(B(Sp(4)/\mathbb{Z}_2); \mathbb{Z}_2) \) and \( H^{\ast}(B(SU(8)/\mathbb{Z}_2); \mathbb{Z}_2) \).

First, for \( H^{\ast}(B(SU(8)/\mathbb{Z}_2); \mathbb{Z}_2) \), since \( B g^{\ast}(w_4) = c_2 \), \( B g^{\ast}(w_6) = c_3 \), and \( B g^{\ast}(w_{10}) = c_5 \), it follows by the naturality that \( c_2 \), \( c_3 \), and \( c_5 \) survive to the \( E_{\infty} \)-page of the Leray–Serre spectral sequence. Consequently, \( B \widetilde{g}^{\ast}(y_4) \), \( B \widetilde{g}^{\ast}(y_6) \), and \( B \widetilde{g}^{\ast}(y_{10}) \) map to generators of \( H^{\ast}(B(SU(8)/\mathbb{Z}_2); \mathbb{Z}_2) \).

Furthermore, since \( B g^{\ast \leq 11} \) is surjective, the naturality implies that \( \alpha \), \( \beta \), \( \gamma \), and \( \delta \) also survive to the \( E_{\infty} \)-page of the Leray–Serre spectral sequence. Thus, \( B \widetilde{g}^{\ast}(x_2) \), \( B \widetilde{g}^{\ast}(x_3) \), \( B \widetilde{g}^{\ast}(x_5) \), and \( B \widetilde{g}^{\ast}(x_9) \) map to generators of \( H^{\ast}(B(SU(8)/\mathbb{Z}_2); \mathbb{Z}_2) \).

We denote these generators by \( y_4 \), \( y_6 \), \( y_{10} \), \( x_2 \), \( x_3 \), \( x_5 \), and \( x_9 \), using the same notation as in the case of \( H^{\ast}(BSs(16); \mathbb{Z}_2) \).

The 9-dimensional relation is mapped to
\[
B \widetilde{g}^{\ast}((x_5 + x_2 x_3)y_4 + x_3 y_6 + x_2 y_7) = (x_5 + x_2 x_3)y_4 + x_3 y_6 + x_2 Sq^1(y_6),
\]
and this is canceled by \( c_4 \) through the differential map in the Leray–Serre spectral sequence.

By considering a similar argument for \( H^{\ast}(B(Sp(4)/\mathbb{Z}_2); \mathbb{Z}_2) \), the following conclusions about the cohomology rings are drawn:

\begin{align*}
H^{\ast\leq 11}(B(Sp(4)/\Z_2); \Z_2) &= \Z_2[x_2, x_3, x_5, x_9, y_4]/((x_5+x_2x_3)y_4+x_3Sq^2(y_4)+x_2Sq^3(y_4)),\\
H^{\ast\leq 11}(B(SU(8)/\Z_2); \Z_2) &= \Z_2[x_2, x_3, x_5, x_9, y_4, y_6, y_{10}]/((x_5 + x_2 x_3)y_4 + x_3 y_6 + x_2 Sq^1(y_6)), \\
H^{\ast\leq 11}(BSs(16); \Z_2) &= \Z_2[x_2, x_3, x_5, x_9, y_4, y_6, y_7, y_{10}, y_{11}]/((x_5+x_2x_3)y_4+x_3y_6+x_2y_7).
\end{align*}

Next, we determine the destination of each element under the Steenrod action. An important lemma is as follows.

\begin{lemma}  
\[
 Sq^2(y_4) = x_2y_4 \in H^{\ast}(B(Sp(4)/\mathbb{Z}_2); \mathbb{Z}_2).
\]
\end{lemma}

\begin{proof}
   This follows from the use of the following commutative diagram.

\[
\begin{tikzcd}
SO(5) \arrow[rr, hook, "diagonal"] & \arrow[d, "\circlearrowleft", phantom]& SO(5)\times SO(5) \\
Sp(2)/\mathbb{Z}_2 \arrow[r, hook, "diagonal"] \arrow[u, "\simeq"'] & (Sp(2)\times Sp(2))/{\{\pm (I, I)\}} \arrow[r, two heads, "pr"] \arrow[d, hook]& (Sp(2)/\mathbb{Z}_2)\times (Sp(2)/\mathbb{Z}_2) \arrow[u, "\simeq"'] \\
& Sp(4)/\mathbb{Z}_2 
\end{tikzcd}
\]

Here, since \( Sp(2) \) and \( Spin(5) \) are isomorphic as Lie groups, \( Sp(2)/\mathbb{Z}_2 \) and \( SO(5) \) are also isomorphic as Lie groups. In the following, we identify these two Lie groups.

First, we have the following cohomology rings:
\begin{align*}
H^{\ast \leq 6}(BSp(2); \mathbb{Z}_2) &= \mathbb{Z}_2[p_1],\\
H^{\ast \leq 6}(B^2\mathbb{Z}_2; \mathbb{Z}_2) &= \mathbb{Z}_2[\alpha, \beta, \gamma],\\
H^{\ast}(B(Sp(2)/\mathbb{Z}_2); \mathbb{Z}_2) &= H^{\ast}(BSO(5); \mathbb{Z}_2) = \mathbb{Z}_2[w_2, w_3, w_4, w_5].\\
\end{align*}

Considering the Leray-Serre spectral sequence associated with the fibration 

\[
\begin{tikzcd}
BSp(2) \arrow[r, "a"] & B(Sp(2)/\mathbb{Z}_2) \arrow[r, "b"] & B^2\mathbb{Z}_2, 
\end{tikzcd}
\]
we see that \( p_1, \alpha, \beta, \gamma \) all survive and correspond to the respective Stiefel-Whitney classes in their dimensions.

Next, consider the following maps of fibrations. 
\begin{equation}\label{Sp(2):fibration1}
    \begin{tikzcd}
BSp(2) \arrow[r, "a"] \arrow[d, "diagonal"'] & B(Sp(2)/\mathbb{Z}_2) \arrow[r, "b"] \arrow[d, "h"'] & B^2\mathbb{Z}_2 \arrow[d, "id"] \\
BSp(2)\times BSp(2) \arrow[r, "c"] \arrow[d, "id"'] & B((Sp(2)\times Sp(2))/{\{\pm (I, I)\}}) \arrow[r, "d"] \arrow[d, "i"'] & B^2\mathbb{Z}_2 \arrow[d, "diagonal"] \\
BSp(2)\times BSp(2) \arrow[r, "a\times a"] & B(Sp(2)/\mathbb{Z}_2)\times B(Sp(2)/\mathbb{Z}_2) \arrow[r, "b\times b"] & B^2\mathbb{Z}_2\times B^2\mathbb{Z}_2
\end{tikzcd}
\end{equation}

In this case, as previously discussed, the Leray-Serre spectral sequence for the bottom fibration has trivial differential maps in dimensions up to 6. By naturality, the Leray-Serre spectral sequences for all fibrations collapse at the \( E_2 \)-page within the range in dimensions up to 6.

Consider \( w_4 + w_4' \in H^4(B(Sp(2)/\mathbb{Z}_2) \times B(Sp(2)/\mathbb{Z}_2); \mathbb{Z}_2) \), where the parts related to the second component are written with a prime (\('\)). We have 
\[
c^{\ast} \circ i^{\ast}(w_4 + w_4') = (a \times a)^{\ast}(w_4 + w_4') = p_1 + p_1'.
\]

Here, for the following maps of fibrations
\begin{equation}\label{Sp(2):fibration2}
    \begin{tikzcd}
BSp(2)\times BSp(2) \arrow[r, "c"] \arrow[d, "j"'] & B((Sp(2)\times Sp(2))/{\{\pm (I, I)\}}) \arrow[r, "d"] \arrow[d, "k"'] & B^2\mathbb{Z}_2 \arrow[d, "id"'] \\
BSp(4) \arrow[r, "\lambda"] & B(Sp(4)/\mathbb{Z}_2) \arrow[r, "\phi"] & B^2\mathbb{Z}_2,
\end{tikzcd}
\end{equation}
since the Leray-Serre spectral sequences for both fibrations collapse at the \( E_2 \)-page within the range up to dimensions 6, the morphism between exact sequences is constructed as follows.

\[
\begin{tikzcd}[column sep=0.7em, row sep=2em]
0 & H^4(BSp(2)\times BSp(2); \mathbb{Z}_2) \arrow[l, ""] \arrow[rd, "\circlearrowleft", phantom]& H^4(B((Sp(2)\times Sp(2))/{\{\pm (I, I)\}}); \mathbb{Z}_2) \arrow[l, "c^{\ast}"] \arrow[rd, "\circlearrowleft", phantom] & H^4(B^2\mathbb{Z}_2; \mathbb{Z}_2) \arrow[l, "d^{\ast}"] & 0 \arrow[l, ""] \\
0 & H^4(BSp(4); \mathbb{Z}_2) \arrow[l, ""] \arrow[u, "j^{\ast}"] & H^4(B(Sp(4)/\mathbb{Z}_2); \mathbb{Z}_2) \arrow[l, "\lambda^{\ast}"] \arrow[u, "k^{\ast}"] & H^4(B^2\mathbb{Z}_2; \mathbb{Z}_2) \arrow[l, "\phi^{\ast}"] \arrow[u, "id"] & 0 \arrow[l, ""]
\end{tikzcd}
\]

Here, since \( j^{\ast}(1 + p_1 + p_2 + p_3 + p_4) = (1 + p_1 + p_2)(1 + p_1' + p_2') \), we have \( j^{\ast}(p_1) = p_1 + p_1' \). Therefore, 
\[
c^{\ast}(k^{\ast}(y_4)) = j^{\ast} \circ \lambda^{\ast}(y_4) = j^{\ast}(p_1) = p_1 + p_1' = c^{\ast}(i^{\ast}(w_4 + w_4')),
\]
which implies 
\[
k^{\ast}(y_4) - i^{\ast}(w_4 + w_4') \in \ker c^{\ast} = \mathrm{Im}\, d^{\ast}.
\]

If \( k^{\ast}(y_4) - i^{\ast}(w_4 + w_4') = d^{\ast}(\alpha^2) \) holds, then 
\[
(k \circ h)^{\ast}(y_4) = h^{\ast}(i^{\ast}(w_4 + w_4') + d^{\ast}(\alpha^2)) = (i \circ h)^{\ast}(w_4 + w_4') + b^{\ast}(\alpha)^2.
\]
Considering that \( i \circ h \) is the diagonal map and \( b^{\ast} \) maps 2-dimensional generators to 2-dimensional generators, we obtain 
\[
(k \circ h)^{\ast}(y_4) = w_2^2.
\]

However, since \( (x_5 + x_2x_3)y_4 + x_3Sq^2(y_4) + x_2Sq^3(y_4) = 0 \) in \( H^4(B(Sp(4)/\mathbb{Z}_2); \mathbb{Z}_2) \), we must have 
\[
(k \circ h)^{\ast}((x_5 + x_2x_3)y_4 + x_3Sq^2(y_4) + x_2Sq^3(y_4)) = 0.
\]
On the other hand, using \( (k \circ h)^{\ast}(x_2) = (k \circ h)^{\ast}(\phi^{\ast}(\alpha)) = b^{\ast}(\alpha) = w_2 \) and Wu's formula, we calculate the left-hand side as follows:
\[
\text{(LHS)} = (Sq^2(Sq^1(w_2)) + w_2Sq^1(w_2))w_2^2 + Sq^1(w_2)Sq^2(w_2^2) + w_2Sq^3(w_2^2) = w_5w_2^2 + w_3^3 \neq 0,
\]
which leads to a contradiction. Therefore, we conclude that 
\begin{equation}\label{ky_4}
    k^{\ast}(y_4) = i^{\ast}(w_4 + w_4').
\end{equation}

Thus, we have
\begin{align*}
k^{\ast}(Sq^2(y_4)) &= i^{\ast}(Sq^2(w_4 + w_4')) = i^{\ast}(w_2w_4 + w_2'w_4') = i^{\ast}(w_2)i^{\ast}(w_4) + i^{\ast}(w_2')i^{\ast}(w_4') \\
 &= d^{\ast}(\alpha)i^{\ast}(w_4 + w_4') = k^{\ast}(\phi^{\ast}(\alpha))k^{\ast}(y_4) \\
 &= k^{\ast}(x_2)k^{\ast}(y_4) = k^{\ast}(x_2y_4).
\end{align*}

Within the range up to dimension 6, combining the injectivity of \( j^{\ast} \) with the fact that the Leray-Serre spectral sequences collapse at the \( E_2 \)-term, we conclude that \( k^{\ast} \) is injective in this range. 
Therefore, we obtain \( Sq^2(y_4) = x_2y_4 \).
\end{proof}

From this, we have 
\[
B\widetilde{f}^{\ast}(y_6) = B\widetilde{f}^{\ast}(Sq^2(y_4)) = Sq^2(y_4) = x_2y_4 \in H^4(B(Sp(4)/\mathbb{Z}_2); \mathbb{Z}_2),
\]
and hence
\[
B\widetilde{f}^{\ast}(B\widetilde{g}^{\ast}(y_7)) = B\widetilde{f}^{\ast}(B\widetilde{g}^{\ast}(Sq^1(y_6))) = Sq^1(x_2y_4) = x_3y_4,
\]
which uniquely determines \( B\widetilde{g}^{\ast}(y_7) = x_3y_4 \). The images of the remaining generators can be determined by calculating the images of these elements under Steenrod operations.
\end{proof}

\begin{remark}
The next-dimensional relation of the given 9-dimensional relation lies in the 13-dimension by using Leray-Serre spectral sequence carefully. 
Therefore, we can construct the $E_2$-term of the Adams spectral sequence up to 11-dimension from the mentioned table.
However, since my interest lies in lower dimensions, this will not be addressed in this paper.
\end{remark}

\subsection{\texorpdfstring{Calculation of the Spin-$G$ bordism in dimensions up to 7}{Calculation of the Spin-G bordism in dimensions up to 7}}

In this subsection, we run the Adams spectral sequence
for the three Spin-$G$ bordism groups.

\begin{thm}\label{thm:twistedspinbordism}
Up to degree $7$, we have the following twisted Spin bordism groups.
\begin{table}[ht]
\centering
\begin{tabular}{|c|c|c|c|c|c|c|c|c|c|c|c|}
\hline
k & $0$ & $1$ & $2$ & $3$ & $4$ & $5$ & $6$ & $7$ \\ \hline
$\Omega_k^{\Spin\text{-}Sp(4)}$ & $\Z$ & $0$ & $0$ & $0$ & $\Z\oplus \Z$ & $\Z_2\oplus \Z_2$ & $\Z_2\oplus \Z_2$ & $0$ \\ \hline
$\Omega_k^{\Spin\text{-}SU(8)}$ & $\Z$ & $0$ & $0$ & $0$ & $\Z\oplus \Z$ & $\Z_2$ & $\Z\oplus \Z_2$ & $0$ \\ \hline
$\Omega_k^{\Spin\text{-}Spin(16)}$ & $\Z$ & $0$ & $0$ & $0$ & $\Z\oplus \Z$ & $\Z_2$ & $\Z_2$ & $0$ \\ \hline
\end{tabular}
\end{table}
\end{thm}

\begin{proof}

Since \eqref{Sp(4)fibration}, \eqref{SU(8)fibration}, and \eqref{Spin(16)fibration} are all universal coverings, they are nontrivial as $\mathbb{Z}_2$-bundles. Therefore, the elements of $H^2(BH;\mathbb{Z}_2)$ corresponding to these central extensions are nonzero and so are $x_2$ by using theorem \ref{thm:cohomology ring}.

The map $M^{MTSpin}f_{0, x_2} \to M^{ko}f_{0, x_2}$ induces isomorphisms of homotopy groups up to dimension $7$. Thus, for computations of bordism groups up to dimension $7$, it suffices to consider $M^{ko}f_{0, x_2}$. For detailed terminology, see \cite[Section 3.1]{DY23}.

By combining \cite[Theorem 2.28(3)]{DY23} and the results of Theorem \ref{thm:cohomology ring}, the $\mathcal{A}_1$-module structure of $H_{ko}^{\ast}(M^{ko}f_{0, x_2})$ is computed for each of the three Spin-$G$ structures. 

Here, the goal is to compute $\pi_{\ast\leq 7}(M^{ko}f_{0, x_2})$, and we consider replacing $M^{ko}f_{0, x_2}$ with another $\mathcal{A}_1$-module to facilitate the computation. The following table represents, for each of the three Spin-$G$ structures, an $\mathcal{A}_1$-module that satisfies the following conditions:
\begin{itemize}
    \item In dimensions up to $7$, it is identical to $M^{ko}f_{0, x_2}$ as a $\Z_2$-module.
    \item In dimensions up to $8$, the action of $\mathcal{A}_1$ is completely consistent with that on $M^{ko}f_{0, x_2}$.
\end{itemize}

\newpage
\begin{center}
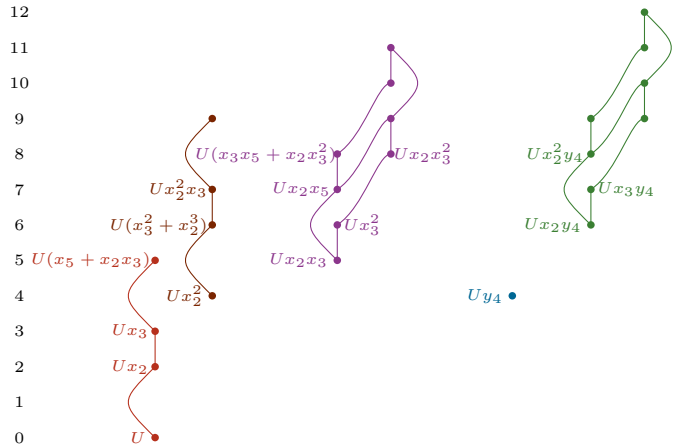
\begin{figure}[ht]
\begin{subfigure}[c]{0.15\textwidth\hskip -110pt}
\begin{tikzpicture}[scale=0.47, every node/.style = {font = \tiny}]
        \foreach \y in {0, ..., 12} {
                \node at (-10.3, \y) {$\y$};
        }
        \begin{scope}[BrickRed]
                \Mzero{-6.5}{0}{$U$};
                \draw (-7.2,2) node {$Ux_2$};
                \draw (-7.2,3) node {$Ux_3$};
                \draw (-8.3,5) node {$U(x_5+x_2x_3)$};
        \end{scope}
        \begin{scope}[Brown]
                \Mzero{-4.9}{4}{$Ux_2^2$};
                \draw (-6.4,6) node {$U(x_3^2+x_2^3)$};
                \draw (-5.85,7) node {$Ux_2^2x_3$};
        \end{scope}
        \begin{scope}[Fuchsia]
                \Aone{-1.4}{5}{$Ux_2x_3$};
                \draw (-0.7,6) node {$Ux_3^2$};
                \draw (-2.4,7) node {$Ux_2x_5$};
                \draw (-3.4,8) node {$U(x_3x_5+x_2x_3^2)$};
                \draw (1,8) node {$Ux_2x_3^2$};
        \end{scope}
        \begin{scope}[MidnightBlue]
                \tikzpt{3.5}{4}{$Uy_4$}{};
        \end{scope}
        \begin{scope}[OliveGreen]
                \Aone{5.7}{6}{$Ux_2y_4$}{};
                \draw (6.7,7) node {$Ux_3y_4$};
                \draw (4.7,8) node {$Ux_2^2y_4$};
        \end{scope}
\end{tikzpicture}
\end{subfigure}
\caption{The case of Spin-$Sp(4)$}
\label{Sp(4)_sseq}
\end{figure}
\end{center}
\begin{center}
\begin{figure}[ht]
\begin{subfigure}[c]{0.15\textwidth\hskip -110pt}
\begin{tikzpicture}[scale=0.47, every node/.style = {font = \tiny}]
        \foreach \y in {0, ..., 12} {
                \node at (-10.3, \y) {$\y$};
        }
        \begin{scope}[BrickRed]
                \Mzero{-6.5}{0}{$U$};
                \draw (-7.2,2) node {$Ux_2$};
                \draw (-7.2,3) node {$Ux_3$};
                \draw (-8.3,5) node {$U(x_5+x_2x_3)$};
        \end{scope}
        \begin{scope}[Brown]
                \Mzero{-4.9}{4}{$Ux_2^2$};
                \draw (-6.4,6) node {$U(x_3^2+x_2^3)$};
                \draw (-5.85,7) node {$Ux_2^2x_3$};
        \end{scope}
        \begin{scope}[Fuchsia]
                \Aone{-1.4}{5}{$Ux_2x_3$};
                \draw (-0.7,6) node {$Ux_3^2$};
                \draw (-2.4,7) node {$Ux_2x_5$};
                \draw (-3.4,8) node {$U(x_3x_5+x_2x_3^2)$};
                \draw (1,8) node {$Ux_2x_3^2$};
        \end{scope}
        \begin{scope}[MidnightBlue]
                \tikzpt{3.5}{4}{$Uy_4$}{};
                \tikzpt{3.5}{6}{$U(y_6+x_2y_4)$}{};
                \sqtwoL(3.5, 4);
        \end{scope}
        \begin{scope}[OliveGreen]
                \Aone{5.7}{6}{$Ux_2y_4$}{};
                \draw (6.7,7) node {$Ux_3y_4$};
                \draw (4.7,8) node {$Ux_2y_6$};
        \end{scope}
\end{tikzpicture}
\end{subfigure}
\caption{The case of Spin-$SU(8)$}
\label{SU(8)_sseq}
\end{figure}
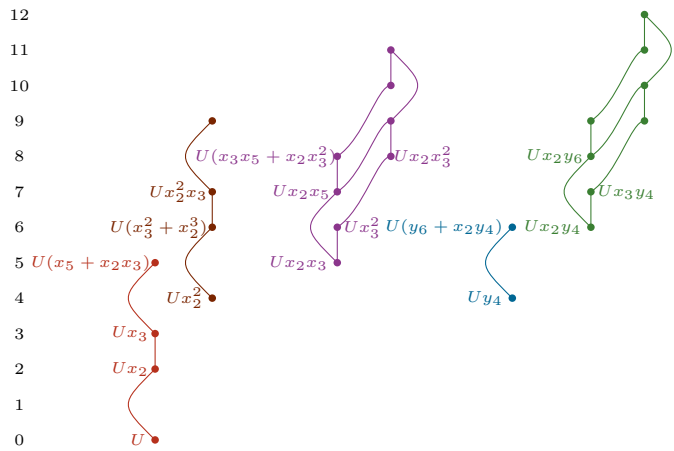
\end{center}
\begin{center}
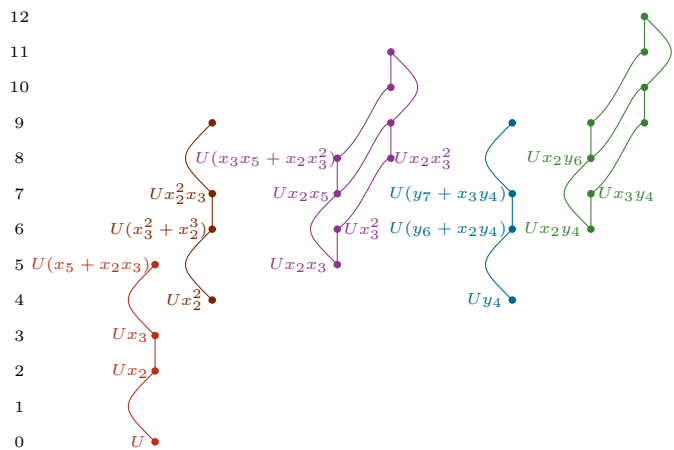
\begin{figure}[ht]
\begin{subfigure}[c]{0.15\textwidth\hskip -110pt}
\begin{tikzpicture}[scale=0.47, every node/.style = {font = \tiny}]
        \foreach \y in {0, ..., 12} {
                \node at (-10.3, \y) {$\y$};
        }
        \begin{scope}[BrickRed]
                \Mzero{-6.5}{0}{$U$};
                \draw (-7.2,2) node {$Ux_2$};
                \draw (-7.2,3) node {$Ux_3$};
                \draw (-8.3,5) node {$U(x_5+x_2x_3)$};
        \end{scope}
        \begin{scope}[Brown]
                \Mzero{-4.9}{4}{$Ux_2^2$};
                \draw (-6.4,6) node {$U(x_3^2+x_2^3)$};
                \draw (-5.85,7) node {$Ux_2^2x_3$};
        \end{scope}
        \begin{scope}[Fuchsia]
                \Aone{-1.4}{5}{$Ux_2x_3$};
                \draw (-0.7,6) node {$Ux_3^2$};
                \draw (-2.4,7) node {$Ux_2x_5$};
                \draw (-3.4,8) node {$U(x_3x_5+x_2x_3^2)$};
                \draw (1,8) node {$Ux_2x_3^2$};
        \end{scope}
        \begin{scope}[MidnightBlue]
                \Mzero{3.5}{4}{$Uy_4$};
                \draw (1.7,6) node {$U(y_6+x_2y_4)$};
                \draw (1.7,7) node {$U(y_7+x_3y_4)$};
        \end{scope}
        \begin{scope}[OliveGreen]
                \Aone{5.7}{6}{$Ux_2y_4$}{};
                \draw (6.7,7) node {$Ux_3y_4$};
                \draw (4.7,8) node {$Ux_2y_6$};
        \end{scope}
\end{tikzpicture}
\end{subfigure}
\caption{The case of Spin-$Spin(16)$}
\label{Spin(16)_sseq}
\end{figure}
\end{center}

These $\mathcal{A}_1$-modules, when applied to $\mathrm{Ext}_{\mathcal{A}_1}(-, \mathbb{Z}_2)$, form the $E_2$-term of the Adams spectral sequence that converges to $\Omega_{\ast}^{\Spin\text{-}G}$. Using the concrete example from Section 4 of \cite{BC18}, we find that the desired $E_2$-term is given in the following table.

\begin{table}[ht]
\centering
\begin{tabular}{c}
\begin{tikzpicture}[scale=0.83]
\begin{axis}[
    xlabel={\(t-s\)}, 
    ylabel={\(s\)},   
    axis lines=middle,
    xtick={0, 1, 2, 3, 4, 5, 6, 7},
    ytick={0, 1, 2, 3, 4, 5, 6},
    xmin=0, xmax=7,
    ymin=0, ymax=6,
    grid=major,
    legend style={at={(0.95,0.95)},anchor=north east}
]

\addplot[only marks, mark=*, mark options={draw=Fuchsia, fill=white, scale=1.5}] coordinates {
    (5,0)
};

\addplot[only marks, mark=*, color=BrickRed] coordinates {
    (0,0) (0,1) (0,2) (0,3) (0,4) (0,5) (0,6)
};

\addplot[only marks, mark=*, color=Brown] coordinates {
    (4,0) (4,1) (4,2) (4,3) (4,4) (4,5) (4,6)
};

\addplot[only marks, mark=*, color=Fuchsia] coordinates {
    (5,0)
};

\addplot[only marks, mark=*, mark options={draw=OliveGreen, fill=white, scale=1.5}] coordinates {
    (6,0)
};

\addplot[only marks, mark=*, color=MidnightBlue] coordinates {
    (4.1,0) (4.1,1) (4.1,2) (4.1,3) (4.1,4) (4.1,5) (4.1,6)
    (5,1) (6,2)
};

\addplot[only marks, mark=*, color=OliveGreen] coordinates {
    (6,0)
};

\draw[-, thick, color=BrickRed] (axis cs:0,0) -- (axis cs:0,1);
\draw[-, thick, color=BrickRed] (axis cs:0,1) -- (axis cs:0,2);
\draw[-, thick, color=BrickRed] (axis cs:0,2) -- (axis cs:0,3);
\draw[-, thick, color=BrickRed] (axis cs:0,3) -- (axis cs:0,4);
\draw[-, thick, color=BrickRed] (axis cs:0,4) -- (axis cs:0,5);
\draw[-, thick, color=BrickRed] (axis cs:0,5) -- (axis cs:0,6);
\draw[-, thick, color=Brown] (axis cs:4,0) -- (axis cs:4,1);
\draw[-, thick, color=Brown] (axis cs:4,1) -- (axis cs:4,2);
\draw[-, thick, color=Brown] (axis cs:4,2) -- (axis cs:4,3);
\draw[-, thick, color=Brown] (axis cs:4,3) -- (axis cs:4,4);
\draw[-, thick, color=Brown] (axis cs:4,4) -- (axis cs:4,5);
\draw[-, thick, color=Brown] (axis cs:4,5) -- (axis cs:4,6);
\draw[-, thick, color=MidnightBlue] (axis cs:4.1,0) -- (axis cs:4.1,1);
\draw[-, thick, color=MidnightBlue] (axis cs:4.1,1) -- (axis cs:4.1,2);
\draw[-, thick, color=MidnightBlue] (axis cs:4.1,2) -- (axis cs:4.1,3);
\draw[-, thick, color=MidnightBlue] (axis cs:4.1,3) -- (axis cs:4.1,4);
\draw[-, thick, color=MidnightBlue] (axis cs:4.1,4) -- (axis cs:4.1,5);
\draw[-, thick, color=MidnightBlue] (axis cs:4.1,5) -- (axis cs:4.1,6);
\draw[-, thick, color=MidnightBlue] (axis cs:4.1,0) -- (axis cs:5,1);
\draw[-, thick, color=MidnightBlue] (axis cs:5,1) -- (axis cs:6,2);
\end{axis}
\end{tikzpicture}
\end{tabular}
\caption{$E_2$-term of $\Omega_{\ast}^{\Spin\text{-}Sp(4)}$.}
\label{tab:Sp(4)E_2}
\end{table}
\begin{table}[ht]
\centering
\begin{tabular}{c}
\begin{tikzpicture}[scale=0.83]
\begin{axis}[
    xlabel={\(t-s\)}, 
    ylabel={\(s\)},   
    axis lines=middle,
    xtick={0, 1, 2, 3, 4, 5, 6, 7},
    ytick={0, 1, 2, 3, 4, 5, 6},
    xmin=0, xmax=7,
    ymin=0, ymax=6,
    grid=major,
    legend style={at={(0.95,0.95)},anchor=north east}
]

\addplot[only marks, mark=*, mark options={draw=Fuchsia, fill=white, scale=1.5}] coordinates {
    (5,0)
};

\addplot[only marks, mark=*, color=BrickRed] coordinates {
    (0,0) (0,1) (0,2) (0,3) (0,4) (0,5) (0,6)
};

\addplot[only marks, mark=*, color=Brown] coordinates {
    (4,0) (4,1) (4,2) (4,3) (4,4) (4,5) (4,6)
};

\addplot[only marks, mark=*, color=Fuchsia] coordinates {
    (5,0)
};

\addplot[only marks, mark=*, mark options={draw=OliveGreen, fill=white, scale=1.5}] coordinates {
    (6,0)
};

\addplot[only marks, mark=*, color=MidnightBlue] coordinates {
    (4.1,0) (4.1,1) (4.1,2) (4.1,3) (4.1,4) (4.1,5) (4.1,6)
    (6,1) (6,2) (6,3) (6,4) (6,5) (6,6) 
};

\addplot[only marks, mark=*, color=OliveGreen] coordinates {
    (6,0)
};

\draw[-, thick, color=BrickRed] (axis cs:0,0) -- (axis cs:0,1);
\draw[-, thick, color=BrickRed] (axis cs:0,1) -- (axis cs:0,2);
\draw[-, thick, color=BrickRed] (axis cs:0,2) -- (axis cs:0,3);
\draw[-, thick, color=BrickRed] (axis cs:0,3) -- (axis cs:0,4);
\draw[-, thick, color=BrickRed] (axis cs:0,4) -- (axis cs:0,5);
\draw[-, thick, color=BrickRed] (axis cs:0,5) -- (axis cs:0,6);
\draw[-, thick, color=Brown] (axis cs:4,0) -- (axis cs:4,1);
\draw[-, thick, color=Brown] (axis cs:4,1) -- (axis cs:4,2);
\draw[-, thick, color=Brown] (axis cs:4,2) -- (axis cs:4,3);
\draw[-, thick, color=Brown] (axis cs:4,3) -- (axis cs:4,4);
\draw[-, thick, color=Brown] (axis cs:4,4) -- (axis cs:4,5);
\draw[-, thick, color=Brown] (axis cs:4,5) -- (axis cs:4,6);
\draw[-, thick, color=MidnightBlue] (axis cs:4.1,0) -- (axis cs:4.1,1);
\draw[-, thick, color=MidnightBlue] (axis cs:4.1,1) -- (axis cs:4.1,2);
\draw[-, thick, color=MidnightBlue] (axis cs:4.1,2) -- (axis cs:4.1,3);
\draw[-, thick, color=MidnightBlue] (axis cs:4.1,3) -- (axis cs:4.1,4);
\draw[-, thick, color=MidnightBlue] (axis cs:4.1,4) -- (axis cs:4.1,5);
\draw[-, thick, color=MidnightBlue] (axis cs:4.1,5) -- (axis cs:4.1,6);
\draw[-, thick, color=MidnightBlue] (axis cs:6,1) -- (axis cs:6,2);
\draw[-, thick, color=MidnightBlue] (axis cs:6,2) -- (axis cs:6,3);
\draw[-, thick, color=MidnightBlue] (axis cs:6,3) -- (axis cs:6,4);
\draw[-, thick, color=MidnightBlue] (axis cs:6,4) -- (axis cs:6,5);
\draw[-, thick, color=MidnightBlue] (axis cs:6,5) -- (axis cs:6,6);

\end{axis}
\end{tikzpicture}
\end{tabular}
\caption{$E_2$-term of $\Omega_{\ast}^{\Spin\text{-}SU(8)}$.}
\label{tab:SU(8)E_2}
\end{table}
\begin{table}[ht]
\centering
\begin{tabular}{c}
\begin{tikzpicture}[scale=0.83]
\begin{axis}[
    xlabel={\(t-s\)}, 
    ylabel={\(s\)},   
    axis lines=middle,
    xtick={0, 1, 2, 3, 4, 5, 6, 7},
    ytick={0, 1, 2, 3, 4, 5, 6},
    xmin=0, xmax=7,
    ymin=0, ymax=6,
    grid=major,
    legend style={at={(0.95,0.95)},anchor=north east}
]

\addplot[only marks, mark=*, mark options={draw=Fuchsia, fill=white, scale=1.5}] coordinates {
    (5,0)
};

\addplot[only marks, mark=*, color=BrickRed] coordinates {
    (0,0) (0,1) (0,2) (0,3) (0,4) (0,5) (0,6)
};

\addplot[only marks, mark=*, color=Brown] coordinates {
    (4,0) (4,1) (4,2) (4,3) (4,4) (4,5) (4,6)
};

\addplot[only marks, mark=*, color=Fuchsia] coordinates {
    (5,0)
};

\addplot[only marks, mark=*, mark options={draw=OliveGreen, fill=white, scale=1.5}] coordinates {
    (6,0)
};

\addplot[only marks, mark=*, color=MidnightBlue] coordinates {
    (4.1,0) (4.1,1) (4.1,2) (4.1,3) (4.1,4) (4.1,5) (4.1,6)
};

\addplot[only marks, mark=*, color=OliveGreen] coordinates {
    (6,0)
};

\draw[-, thick, color=BrickRed] (axis cs:0,0) -- (axis cs:0,1);
\draw[-, thick, color=BrickRed] (axis cs:0,1) -- (axis cs:0,2);
\draw[-, thick, color=BrickRed] (axis cs:0,2) -- (axis cs:0,3);
\draw[-, thick, color=BrickRed] (axis cs:0,3) -- (axis cs:0,4);
\draw[-, thick, color=BrickRed] (axis cs:0,4) -- (axis cs:0,5);
\draw[-, thick, color=BrickRed] (axis cs:0,5) -- (axis cs:0,6);
\draw[-, thick, color=Brown] (axis cs:4,0) -- (axis cs:4,1);
\draw[-, thick, color=Brown] (axis cs:4,1) -- (axis cs:4,2);
\draw[-, thick, color=Brown] (axis cs:4,2) -- (axis cs:4,3);
\draw[-, thick, color=Brown] (axis cs:4,3) -- (axis cs:4,4);
\draw[-, thick, color=Brown] (axis cs:4,4) -- (axis cs:4,5);
\draw[-, thick, color=Brown] (axis cs:4,5) -- (axis cs:4,6);
\draw[-, thick, color=MidnightBlue] (axis cs:4.1,0) -- (axis cs:4.1,1);
\draw[-, thick, color=MidnightBlue] (axis cs:4.1,1) -- (axis cs:4.1,2);
\draw[-, thick, color=MidnightBlue] (axis cs:4.1,2) -- (axis cs:4.1,3);
\draw[-, thick, color=MidnightBlue] (axis cs:4.1,3) -- (axis cs:4.1,4);
\draw[-, thick, color=MidnightBlue] (axis cs:4.1,4) -- (axis cs:4.1,5);
\draw[-, thick, color=MidnightBlue] (axis cs:4.1,5) -- (axis cs:4.1,6);
\end{axis}
\end{tikzpicture}
\end{tabular}
\caption{$E_2$-term of $\Omega_{\ast}^{\Spin\text{-}Spin(16)}$.}
\label{tab:Spin(16)E_2}
\end{table}
\!\!\!\!\!\!\!Here, since the differentials commute with the $\mathrm{Ext}_{\mathcal{A}_1}(\mathbb{Z}_2, \Z_2)$-action \cite[Section 11.3]{DDHM23}, it follows that all differentials vanish by considering the action of multiplication by $h_0$. Therefore, the $E_2$-terms listed above collapse at the $E_2$-page.

Furthermore, Margolis’ theorem \cite{Ma74} tells us that the circled nodes do not participate in any non-trivial extensions and simply split off. Using this, we can compute the three Spin-$G$ bordism groups under investigation.
\end{proof}

\section{Explicit description of generators by using manifolds}
\subsection{Determining the manifold generator at degree 4}

First, we determine the generators of the 4-dimensional part in terms of manifolds.  
Since \( \Omega_4^{\Spin\text{-}Sp(4)} \to \Omega_4^{\Spin\text{-}SU(8)} \to \Omega_4^{\Spin\text{-}Spin(16)} \) are all isomorphisms (Proposition \ref{prop:5-equivalence}), it suffices to find the generators of \( \Omega_4^{\Spin\text{-}Sp(4)} \cong \mathbb{Z} \oplus \mathbb{Z} \).

\begin{prop}  
\begin{equation}
(M, f_M:M\to B(Sp(4)/\Z_2)) \mapsto \left( sign(M)=\dfrac{1}{3}\int_M p_1(M), \int_M f^{\ast}z_4 \right)
\end{equation}
gives the isomorphism \( \Omega_4^{\Spin\text{-}Sp(4)} \cong \mathbb{Z} \oplus \mathbb{Z} \), where $z_4\in H^4(B(Sp(4)/\Z_2);\Z)\cong \Z$ is the generator.    
\end{prop}

\begin{proof}
First, consider the Leray-Serre spectral sequence with \(\mathbb{Z}\)-coefficients for  
\[
\begin{tikzcd}
BSp(4) \arrow[r, ""] & B(Sp(4)/\mathbb{Z}_2) \arrow[r, ""] & B^2\mathbb{Z}_2.
\end{tikzcd}
\]

\[
\begin{array}{c|ccccccc}
5 & 0 & 0 & 0 & 0 & 0 & 0 & 0 \\
4 & p_1 & 0 & 0 & \mu p_1 & 0 & \nu p_1 & \mu^2 p_1 \\
3 & 0 & 0 & 0 & 0 & 0 & 0 & 0 \\
2 & 0 & 0 & 0 & 0 & 0 & 0 & 0 \\
1 & 0 & 0 & 0 & 0 & 0 & 0 & 0 \\
0 & 1 & 0 & 0 & \mu & 0 & \nu & \mu^2 \\ \hline
  & 0 & 1 & 2 & 3 & 4 & 5 & 6\\
\end{array}
\]

Here, \( \mu \in H^3(B^2\mathbb{Z}_2; \Z) \) corresponds to the 2-torsion, and \( \nu \in H^5(B^2\mathbb{Z}_2; \Z) \) corresponds to the 4-torsion \cite{Cl02}.

\begin{lemma}
 $d_5(p_1)=0$
\end{lemma}

\begin{proof}

It suffices to show that $H^5(B(Sp(4)/\mathbb{Z}_2);\mathbb{Z}) \cong 0$ and $H^5(B(Sp(4)/\mathbb{Z}_2);\mathbb{Z}) \cong \mathbb{Z}_2$ are not appropriate. We focus on the following exact Bockstein sequence.

\[
\begin{tikzcd}[column sep=small]
H^2(B(Sp(4)/\Z_2); \mathbb{Z})\arrow[r, "\times 2"]\cong 0 & H^2(B(Sp(4)/\Z_2); \mathbb{Z}) \arrow[r, "mod 2"]\cong 0
& H^2(B(Sp(4)/\Z_2); \mathbb{Z}_2)\arrow[lld, ""]\cong \Z_2\\ 
H^3(B(Sp(4)/\Z_2); \mathbb{Z})\arrow[r, "\times 2 "] \cong \Z_2& H^3(B(Sp(4)/\Z_2); \mathbb{Z}) \arrow[r, "mod 2"]\cong \Z_2
& H^3(B(Sp(4)/\Z_2); \mathbb{Z}_2)\arrow[lld, ""]\cong \Z_2\\ 
H^4(B(Sp(4)/\Z_2); \mathbb{Z})\arrow[r, hook, "\times 2"]\cong \Z & H^4(B(Sp(4)/\Z_2); \mathbb{Z}) \arrow[r, "mod 2"]\cong \Z
& H^4(B(Sp(4)/\Z_2); \mathbb{Z}_2)\arrow[lld, ""]\cong \Z_2\oplus \Z_2\\ 
H^5(B(Sp(4)/\Z_2); \mathbb{Z})\arrow[r, "\times 2"] & H^5(B(Sp(4)/\Z_2); \mathbb{Z}) \arrow[r, "mod 2"]
& H^5(B(Sp(4)/\Z_2); \mathbb{Z}_2) 
\end{tikzcd}
\]

By tracing the exact sequence, we see that the kernel of 
\[
H^4(B(Sp(4)/\mathbb{Z}_2);\mathbb{Z}_2) \to H^5(B(Sp(4)/\mathbb{Z}_2);\mathbb{Z})
\]
is isomorphic to $\mathbb{Z}_2$. This implies that $H^5(B(Sp(4)/\mathbb{Z}_2);\mathbb{Z}) \cong 0$ is not appropriate.

If we assume $H^5(B(Sp(4)/\mathbb{Z}_2);\mathbb{Z}) \cong \mathbb{Z}_2$, then for an element not in the kernel of 
\[
H^4(B(Sp(4)/\mathbb{Z}_2);\mathbb{Z}_2) \to H^5(B(Sp(4)/\mathbb{Z}_2);\mathbb{Z}),
\]
 its image under $Sq^1$ would be a nonzero element. However, since $Sq^1(x_2^2) = Sq^1(y_4) = 0$, this assumption is also not valid.
\end{proof}

Thus, the generator $z_4$ of $H^4(B(Sp(4)/\mathbb{Z}_2);\mathbb{Z}) \cong \mathbb{Z}$ maps to $p_1 \in H^4(BSp(4);\mathbb{Z})\cong \Z$ by pull-back of $BSp(4)\to B(Sp(4)/\mathbb{Z}_2)$. 

We will now show that the homomorphism \( \Omega_4^{\Spin\text{-}Sp(4)} \to \mathbb{Z} \oplus \mathbb{Z} \) stated in the proposition is surjective.
First, consider the following map:  
\begin{equation}
    \begin{tikzcd}
f_{\HP^1}:\HP^1 \arrow[r, hook, ""] & \HP^{\infty} = BSp(1) \arrow[r, ""] & BSp(4) \arrow[r, ""] & B(Sp(4)/\mathbb{Z}_2).
\end{tikzcd}
\end{equation}
Here, the map $BSp(1) \to BSp(4)$ is induced by the natural inclusion $Sp(1) \to Sp(4)$.  

Since $H^2(\HP^1;\mathbb{Z}_2) = 0$, $f_{\HP^1}^{\ast}x_2=w_2(\HP^1)=0$ and this provides an element $(\HP^1, f_{\HP_1}:\HP^1\to B(Sp(4)/\mathbb{Z}_2))$ of $\Omega_4^{\Spin\text{-}Sp(4)}$. As $sign(\HP^1) = 0$ and $z_4$ pulls back to the 4-dimensional generator, we have $\int_{\HP^1} f_{\HP^1}^{\ast}z_4 = 1$. Therefore, this element maps to $(0,1)$ by the homomorphism \( \Omega_4^{\Spin\text{-}Sp(4)} \to \mathbb{Z} \oplus \mathbb{Z} \).

Next, consider the following map:  
\[
\begin{tikzcd}
f_{\CP^2}:\CP^2 \arrow[r, hook, ""] & \CP^{\infty} = BSO(2) \arrow[r, ""] & BSO(3) = B(Sp(1)/\mathbb{Z}_2) \arrow[r, ""] & B(Sp(4)/\mathbb{Z}_2).
\end{tikzcd}
\]  
Here, the map $BSO(2) \to BSO(3)$ is induced by the natural inclusion $SO(2) \to SO(3)$, and the map \( B(Sp(1)/\mathbb{Z}_2) \to B(Sp(4)/\mathbb{Z}_2) \) is induced by the map \( Sp(1)/\mathbb{Z}_2 \to Sp(4)/\mathbb{Z}_2 \), which is the map sending the equivalence class of a matrix \( [A] \) to \( [A \oplus A \oplus A \oplus A] \).

\begin{equation}\label{Sp(1):fibration1}
    \begin{tikzcd}
BSp(1) \arrow[r, ""] \arrow[d, "B(\text{fourfold 
block sum})"'] & B(Sp(1)/\Z_2) \arrow[r, ""] \arrow[d, ""'] & B^2\Z_2 \arrow[d, "id"] \\
BSp(4)\arrow[r, "\lambda"] & B(Sp(4)/\Z_2) \arrow[r, "\phi"]  & B^2\Z_2
\end{tikzcd}
\end{equation}

By naturality, the pull-back of \( x_2 = \phi^{\ast}(\alpha) \) is mapped to the two-dimensional generator \( w_2 \in H^2(BSO(3); \mathbb{Z}_2)=H^2(B(Sp(1)/\Z_2); \mathbb{Z}_2)\). Therefore, \( f_{\CP^2}^{\ast}x_2 \) is the two-dimensional generator $ c_1 \in H^2(\mathbb{CP}^2; \mathbb{Z}_2) \cong \Z_2$, which coincides with \( w_2(\mathbb{CP}^2) \). Thus, this provides an element $(\CP^2, f_{\CP^2}:\CP^2 \to B(Sp(4)/\Z_2) )$ of $\Omega_4^{\text{Spin-}Sp(4)}$.

Since $sign(\CP^2) = 1$, this element is mapped to $(1, \ast)$ by the homomorphism \( \Omega_4^{\Spin\text{-}Sp(4)} \to \mathbb{Z} \oplus \mathbb{Z} \).
The elements $(0,1)$ and $(1, \ast)$ generate $\mathbb{Z} \oplus \mathbb{Z}$, so the homomorphism given in the proposition is surjective. 
By the Hopfian property of $\Z\oplus \Z$, the claim is proved.
\end{proof}
\begin{remark}
By calculating the image of \( (\mathbb{CP}^2, f_{\mathbb{CP}^2}:\mathbb{CP}^2 \to B(Sp(4)/\mathbb{Z}_2)) \) under the homomorphism \( \Omega_4^{\Spin\text{-}Sp(4)} \to \mathbb{Z} \oplus \mathbb{Z} \), we find that it is mapped to \( (1, 1) \).
\end{remark}
From the above proposition, we also see that $\HP^1$ and $\CP^2$ provide the generators of $\Omega_4^{\text{Spin-}Sp(4)}$.

\subsection{Determining the manifold generator at degree 5}

First, we consider the blue $\mathbb{Z}_2 \in \Omega_5^{\text{Spin-}Sp(4)}$ in Table \ref{tab:Sp(4)E_2}. This element is obtained by multiplying the blue $\mathbb{Z} \in \Omega_4^{\text{Spin-}Sp(4)}$ by $S^1$ which is non-zero element in $\Omega_1^{\text{Spin}}$ \cite[section 11.4]{DDHM23}. Therefore, we investigate the blue part of $\Omega_4^{\text{Spin-}Sp(4)}$ in Table \ref{tab:Sp(4)E_2}. Since this blue part corresponds to $y_4 \in H^4(B(Sp(4)/\mathbb{Z}_2);\mathbb{Z}_2)$, we need to find an element $(M, f_M:M\to B(Sp(4)/\Z_2)) \in \Omega_4^{\text{Spin-}Sp(4)}$ such that $\int_M f_M^{\ast}y_4 = 1$.  

Pull-back of $BSp(4) \to B(Sp(4)/\mathbb{Z}_2)$ maps $y_4$ to $p_1 \in H^4(BSp(4);\mathbb{Z}_2)$. From the previous subsection, we see that $(\HP^1, f_{\HP^1}:\HP^1\to B(Sp(4)/\Z_2))$ satisfies this condition. Thus, the blue $\mathbb{Z}_2 \in \Omega_5^{\text{Spin-}Sp(4)}$ corresponds to $\HP^1 \times S^1$.

Next, we consider the purple $\mathbb{Z}_2$ in Table \ref{tab:Sp(4)E_2}. This element is characterized by an element $(M, f_M:M\to B(Sp(4)/\Z_2))$ such that $\int_M f_M^{\ast}(x_2x_3) = 1$.  

Here, we analyze the Wu manifold $f_{SU(3)/SO(3)}:SU(3)/SO(3) \to BSO(3) = B(Sp(1)/\mathbb{Z}_2) \to B(Sp(4)/\mathbb{Z}_2)$. By the naturality of the fibration \eqref{Sp(1):fibration1}, $x_2$ and $x_3$ map to the 2-dimensional generator $w_2$ and the 3-dimensional generator $w_3$ of $H^{\ast}(BSO(3);\mathbb{Z}_2)$, respectively. Thus, Debray and Yu \cite[section 4.4]{DY24} say that $(SU(3)/SO(3), f_{SU(3)/SO(3)}:SU(3)/SO(3) \to B(Sp(4)/\mathbb{Z}_2)$ is  the desired object.  

From the above, we conclude that the purple $\mathbb{Z}_2 \in \Omega_5^{\text{Spin-}Sp(4)}, \Omega_5^{\text{Spin-}SU(8)}, \Omega_5^{\text{Spin-}Spin(16)}$ correspond to the Wu manifold $SU(3)/SO(3)$.

\subsection{Determining the Manifold Generator at degree 6}

One of the generators of the direct summand $\Omega_6^{\text{Spin-}Sp(4)} \cong \mathbb{Z}_2 \oplus \mathbb{Z}_2$ is given by $\HP^1 \times S^1 \times S^1$, as in the previous subsection. The other generator of the direct summand is characterized by an element $(M, f_M:M\to B(Sp(4)/\Z_2))$ such that $\int_M f_M^{\ast}(x_2y_4) = 1$.  

We show that the following map provides this generator:
\[
\begin{tikzcd}[column sep=small]
 & \CP^2 \times \CP^1 \arrow[d, "c_1 \times c_1' \times (c_1 + c_1')"']  & \\
 & B^2\mathbb{Z} \times B^2\mathbb{Z} \times B^2\mathbb{Z} = BSO(2) \times B(SO(2) \times SO(2)) \arrow[d] & \\
 & BSO(2) \times BSO(4) \arrow[d] & \\
B(Sp(2) \times Sp(2)/{\pm(I, I)}) \arrow[r, "i"] \arrow[d, "k"'] & BSO(5) \times BSO(5) = B(Sp(2)/\mathbb{Z}_2) \times B(Sp(2)/\mathbb{Z}_2) \arrow[r] & B^2\mathbb{Z}_2 \\
B(Sp(4)/\mathbb{Z}_2) & 
\end{tikzcd}
\]  
Here, $c_1$ and $c_1'$ are the generators corresponding to the first and second components of $H^2(\CP^2\times \CP^1; \Z)$, respectively. Moreover, $B(SO(2) \times SO(2)) \to BSO(4)$, $BSO(2) \to BSO(5)$, and $BSO(4) \to BSO(5)$ are induced by natural inclusions.

Here, we show that the map $\CP^2 \times \CP^1 \to BSO(5) \times BSO(5)$ admits a homotopy lift via $i$. 

\[
\begin{tikzcd}[column sep=small]
 & \CP^2 \times \CP^1 \arrow[d, "\theta"']\arrow[ld, dashed, "\exists \eta"'] \arrow[rd, "0"]  & \\
B(Sp(2) \times Sp(2)/{\pm(I, I)}) \arrow[r, "i"] \arrow[d, "k"'] & BSO(5) \times BSO(5) = B(Sp(2)/\mathbb{Z}_2) \times B(Sp(2)/\mathbb{Z}_2) \arrow[r] & B^2\mathbb{Z}_2 \\
B(Sp(4)/\mathbb{Z}_2) & 
\end{tikzcd}
\]  

Since $i^{\ast}(w_2 + w_2') = i^{\ast}((b \times b)^{\ast}(\alpha + \alpha')) = d^{\ast}(w_2 + w_2) = 0$ \eqref{Sp(2):fibration1}, it follows that the obstruction class for lifting the map $i$ is $w_2 + w_2' \in H^2(BSO(5) \times BSO(5); \mathbb{Z}_2)$.

By the pullback of $\theta$, $w_2$ maps to $c_1 \in H^2(\CP^2 \times \CP^1; \mathbb{Z}_2)$, and $w_2'$ maps to $c_1' + (c_1 + c_1') = c_1 \in H^2(\CP^2 \times \CP^1; \mathbb{Z}_2)$. Thus, the obstruction $w_2 + w_2'$ vanishes, and consequently, a homotopy lift $\eta$ exists. So we define $f_{\CP^2\times \CP^1}:\CP^2\times \CP^1\to B(Sp(4)/\Z_2)$ as $k \circ \eta$.

Since
\[
f_{\CP^2\times \CP^1}^{\ast}x_2=(k \circ \eta)^{\ast}(x_2) = (\eta^{\ast}\circ k^{\ast})(\phi^{\ast}(\alpha)) = \eta^{\ast}(d^{\ast}(\alpha)) = \eta^{\ast}(i^{\ast} w_2) = \theta^{\ast}w_2 = c_1,
\]
$w_2(\CP^2) = c_1$, and $w_2(\CP^1) = 0$, $f_{\CP^2\times \CP^1}^{\ast}x_2$ agrees with $w_2(\CP^2 \times \CP^1)$. Therefore, this defines an element $(\CP^2\times \CP^1, f_{\CP^2\times \CP^1}:\CP^2\times \CP^1\to B(Sp(4)/\Z_2))$ of $\Omega_6^{\text{Spin-}Sp(4)}$. 

From \eqref{ky_4},
\[
f_{\CP^2\times \CP^1}^{\ast}(y_4) = \eta^{\ast}(k^{\ast}(y_4)) = \eta^{\ast}(i^{\ast}(w_4 + w_4')) = \theta^{\ast} w_4+\theta^{\ast} w_4'=0 + c_1'(c_1 + c_1') = c_1c_1',
\]
\[
\int_{\CP^2 \times \CP^1} f_{\CP^2\times \CP^1}^{\ast}(x_2y_4) = \int_{\CP^2 \times \CP^1} c_1 \cdot c_1c_1' = \int_{\CP^2 \times \CP^1} c_1^2c_1' = 1,
\]
this satisfies the desired property.

This element also survives in $\Omega_6^{\text{Spin-}SU(8)}$ and $\Omega_6^{\text{Spin-}Spin(16)}$. Therefore, it remains to determine the generators of $\Omega_6^{\text{Spin-}SU(8)} \cong \mathbb{Z} \oplus \mathbb{Z}_2$.

\begin{lemma}
\[
(M, f_M:M\to B(SU(8)/\Z_2)) \mapsto \left( \frac{1}{2} \int_M f_M^{\ast}z_6, \int_M f_M^{\ast}(x_2y_4) \right)
\]
gives the isomorphism \( \Omega_6^{\Spin\text{-}SU(8)} \cong \mathbb{Z} \oplus \mathbb{Z}_2 \), where $z_6 \in H^6(B(SU(8)/\mathbb{Z}_2);\mathbb{Z}) \cong \mathbb{Z} \oplus \mathbb{Z}_2$ is the torsion-free generator.
\end{lemma}

\begin{proof}
First, we show that this defines a homomorphism. For this, it suffices to show that $\int_M f_M^\ast z_6$ is always even, which is verified by computing the mod 2 reduction of $z_6$.

Consider the following fibration:
\[
\begin{tikzcd}
BSU(8) \arrow[r, ""] & B(SU(8)/\mathbb{Z}_2) \arrow[r, ""] & B^2\mathbb{Z}_2.
\end{tikzcd}
\]
The Leray-Serre spectral sequence for both $\mathbb{Z}$ and $\mathbb{Z}_2$ coefficients appears as follows:

\begin{equation}\label{BSU(8) Z coefficient}
    \begin{array}{c|cccccccc}
6 & c_3 &  &  & & & & & \\
5 & 0 & 0 & 0 &  &  &  &  & \\
4 & c_2 & 0 & 0 & \mu c_2 &  &  &  & \\
3 & 0 & 0 & 0 & 0 & 0 &  & & \\
2 & 0 & 0 & 0 & 0 & 0 & 0 &  & \\
1 & 0 & 0 & 0 & 0 & 0 & 0 & 0 & \\
0 & 1 & 0 & 0 & \mu & 0 & \nu & \mu^2 & \xi \\ \hline
  & 0 & 1 & 2 & 3 & 4 & 5 & 6 & 7\\
\end{array}
\end{equation}
\begin{equation}\label{BSU(8) Z2 coefficient}
    \begin{array}{c|cccccccc}
6 & c_3 & &  &  &  &  & &  \\
5 & 0 & 0 & 0 &  &  &  & & \\
4 & c_2 & 0 & \alpha c_2 & \beta c_2 &  &  &  & \\
3 & 0 & 0 & 0 & 0 & 0 &  &  & \\
2 & 0 & 0 & 0 & 0 & 0 & 0 &  & \\
1 & 0 & 0 & 0 & 0 & 0 & 0 & 0 & \\
0 & 1 & 0 & \alpha & \beta & \alpha^2 & \alpha \beta,\gamma & \alpha^3, \beta^2 & \alpha^2\beta, \alpha \gamma \\ \hline
  & 0 & 1 & 2 & 3 & 4 & 5 & 6 & 7\\
\end{array}
\end{equation}

As mentioned at theorem \ref{thm:cohomology ring}, the $\mathbb{Z}_2$-coefficient Leray-Serre spectral sequence has no nontrivial differential maps in dimensions up to 8, so it collapses at the $E_2$ page in this range. Furthermore, by following the Bockstein long exact sequence, we find that the mod 2 reductions of $\mu$ and $\xi$ are $\beta$ and $\alpha^2 \beta = Sq^1(\alpha^3)$, respectively.

Since $\beta$ and $\alpha^2 \beta$ survive to the $E_\infty$ page of \eqref{BSU(8) Z2 coefficient}, $\mu$ and $\xi$ also survive to the $E_\infty$ page of \eqref{BSU(8) Z coefficient}. Therefore, the only possible nontrivial differential map is $d_4(c_2)$, which, as in the case of $Sp(4)$, is also zero (though we do not rely on this fact here).

From this, we conclude that \( H^6(B(SU(8)/\mathbb{Z}_2); \mathbb{Z}) \cong \mathbb{Z} \oplus \mathbb{Z}_2 \), and the mod 2 reduction of the generator of the free part is given by \( y_6 + \epsilon_1 x_2 y_4 + \epsilon_2 x_3^2 + \epsilon_3 x_2^3 \), where $\epsilon_1, \epsilon_2, \epsilon_3$ are either $0$ or $1$. By considering the Bockstein long exact sequence, we see that any element obtained via mod 2 reduction must vanish under $Sq^1$. 

Since
\[
Sq^1(y_6 + \epsilon_1 x_2 y_4 + \epsilon_2 x_3^2 + \epsilon_3 x_2^3) = (1 + \epsilon_1)x_3 y_4 + \epsilon_3 x_2^2 x_3,
\]
we deduce that $\epsilon_1 = 1$ and $\epsilon_3 = 0$. By appropriately adding $\mu^2$, we obtain a generator of the free part of \( H^6(B(SU(8)/\mathbb{Z}_2); \mathbb{Z}) \cong \mathbb{Z} \oplus \mathbb{Z}_2 \) whose mod 2 reduction is \( y_6 + x_2 y_4 \), and we denote this generator by $z_6$.

To show that $\int_M f_M^\ast z_6$ is always even, it suffices to prove that there does not exist \( (M, f_M:M\to B(SU(8)/\Z_2)) \in \Omega_6^{\text{Spin-}SU(8)} \) such that 
\[
\int_M f_M^{\ast}(y_6 + x_2y_4) = 1.
\]

Using the properties of the Wu class, we have
\[
\int_M f_M^{\ast}y_6 = \int_M Sq^2(f_M^{\ast}y_4) = \int_M v_2(M)f_M^{\ast}y_4.
\]

Since \( v_2(M) = w_2(M) + w_1(M)^2 = w_2(M) = f_M^{\ast}x_2 \), it follows that
\[
\int_M f_M^{\ast}y_6 = \int_M f_M^{\ast}(x_2y_4).
\]

Thus, we always have
\[
\int_M f_M^{\ast}(y_6 + x_2y_4) = 0,
\]
proving the claim.

Next, we show that the homomorphism \( \Omega_6^{\Spin\text{-}SU(8)} \to \mathbb{Z} \oplus \mathbb{Z}_2 \) is surjective. Consider the following diagram:
\[
\begin{tikzcd}[column sep=small]
 & \mathbb{CP}^1 \times \mathbb{CP}^1 \times \mathbb{CP}^1 \arrow[d, "c_1 \times (c_1'+c_1'') \times \{-(c_1 + c_1'+c_1'')\}"']  & \\
 & B^2\mathbb{Z} \times B^2\mathbb{Z} \times B^2\mathbb{Z} = B(U(1) \times U(1) \times U(1)) \arrow[d] & \\
 & BU(3) \arrow[d] & \\
BSU(8) \arrow[r, ""] \arrow[d, ""'] & BU(8) \arrow[r] & B^2\mathbb{Z} \\
B(SU(8)/\mathbb{Z}_2) & 
\end{tikzcd}
\]

Here, \( c_1 \), \( c_1' \), and \( c_1'' \in H^2(\mathbb{CP}^1 \times \mathbb{CP}^1 \times \mathbb{CP}^1 ; \Z) \) denote the generators of the first, second, and third components, respectively. Moreover, $B(U(1) \times U(1) \times U(1)) \to BU(3)$ and $BU(3) \to BU(8)$ are induced by natural inclusion maps.

\[
\begin{tikzcd}[column sep=small]
 & \mathbb{CP}^1 \times \mathbb{CP}^1 \times \mathbb{CP}^1 \arrow[d, "\kappa"'] \arrow[ld, dashed, "\exists \iota"'] \arrow[rd, "0"]  & \\
BSU(8) \arrow[r, ""] \arrow[d, ""'] & BU(8) \arrow[r, ""] & B^2\mathbb{Z} \\
B(SU(8)/\mathbb{Z}_2) & 
\end{tikzcd}
\]

The obstruction class of \( BSU(8) \to BU(8) \) is \( c_1 \), and \( c_1 \) maps to 
\[
c_1 + (c_1' + c_1'') + \{ -(c_1 + c_1' + c_1'') \} = 0
\]
by the pullback of $\kappa$, which ensures the existence of a homotopy lift $\iota$.

By the pullback of \( BSU(8) \to B(SU(8)/\mathbb{Z}_2) \), \( x_2 \) maps to \( 0 \), and \( z_6 \) maps to \( c_3 \). Hence, we have 
\[
f_{\CP^1\times \CP^1\times \CP^1}^{\ast}x_2 = 0,  
\]
\[
f_{\CP^1\times \CP^1\times \CP^1}^{\ast}z_6 = \iota^{\ast}c_3 = \kappa^{\ast}c_3 = c_1 \times (c_1' + c_1'') \times \{-(c_1 + c_1' + c_1'')\} = -2c_1c_1'c_1''.
\]
Since \( w_2(\mathbb{CP}^1 \times \mathbb{CP}^1 \times \mathbb{CP}^1) = 0 \), the element $(\mathbb{CP}^1 \times \mathbb{CP}^1 \times \mathbb{CP}^1, f_{\CP^1\times \CP^1\times \CP^1}:\mathbb{CP}^1 \times \mathbb{CP}^1 \times \mathbb{CP}^1\to B(SU(8)/\Z_2))$ defines a class in \( \Omega_6^{\text{Spin-}SU(8)} \), and maps to \( (-1, 0) \) under the given homomorphism. As the previously considered \( \mathbb{CP}^2 \times \mathbb{CP}^1 \) maps to the 2-torsion element \( (0, 1) \), which establishes surjectivity.
By the Hopfian property of $\Z\oplus \mathbb{Z}_2$, the claim is proven.
\end{proof}

In conclusion, the generators of \( \Omega_6^{\text{Spin-}SU(8)} \cong \mathbb{Z} \oplus \mathbb{Z}_2 \) are given by \( \mathbb{CP}^1 \times \mathbb{CP}^1 \times \mathbb{CP}^1 \) and \( \mathbb{CP}^2 \times \mathbb{CP}^1 \).

\bibliography{spinbor}
\bibliographystyle{plain}

\end{document}